\documentclass[a4paper,10pt,final]{siamart1116}
\usepackage{etex} 
\usepackage[utf8]{inputenc}

\usepackage[english]{babel}

\usepackage[sort]{cite}
\usepackage[linesnumbered,ruled,algo2e,algosection]{algorithm2e}
\usepackage{booktabs}
\usepackage{amsmath}
\usepackage{amsfonts}
\usepackage{amssymb}
\usepackage{textcomp}
\usepackage{tikz,pgfplots}
\usepackage[all]{xy}
\usepackage{subcaption}
\usetikzlibrary{backgrounds,patterns,matrix,calc,arrows,positioning,fit,shapes.geometric,decorations.pathmorphing}

\DeclareMathAlphabet{\mathpzc}{OT1}{pzc}{m}{it}

\newcommand{\tuple}[1]{\mathfrak{#1}}

\newcommand{\Var}[1]{\mathcal{#1}}
\newcommand{\Sec}[2]{\sigma_{#1}(#2)}

\newcommand{\tensor}[1]{\mathpzc{#1}}
\newcommand{\vect}[1]{\mathbf{#1}}
\newcommand{\sten}[3]{\vect{#1}_{#2}^{#3}}
\newcommand{\Tang}[2]{\mathrm{T}_{#1} {#2}}

\newcommand{\deriv}[2]{\mathrm{d}_{#1}#2}

\newcommand{\diag}{\operatorname{diag}}

\newcommand{\C}{\mathbb{C}}
\newcommand{\R}{\mathbb{R}}

\newcommand{\reftag}[1]{\ref{#1}}

\newcommand{\reflem}[1]{\cref{#1}}
\newcommand{\refeqn}[1]{{(\ref{#1})}}

\newcommand{\refalg}[1]{Algorithm \ref{#1}}
\newcommand{\refsec}[1]{{\cref{#1}}}

\newcommand{\reffig}[1]{{\cref{#1}}}

\newcommand{\refprop}[1]{{\cref{#1}}}

\newcommand{\refapp}[1]{{\cref{#1}}}

\newcommand{\refrem}[1]{{\cref{#1}}}

\newcommand{\rank}{\operatorname{rank}}

\newtheorem{remark}[theorem]{Remark}

\numberwithin{equation}{section}
\numberwithin{figure}{section}
\numberwithin{table}{section}
\numberwithin{theorem}{section}

\title{A Riemannian trust region method for the\\ canonical tensor rank approximation problem\thanks{Submitted to the editors.\funding{The first author was partially supported by DFG research grant BU 1371/2-2. The second author was supported by a Postdoctoral Fellowship of the Research Foundation--Flanders (FWO).}}
}

\author{
Paul Breiding\thanks{Max-Planck-Institute for Mathematics in the Sciences Leipzig. (\email{breiding@mis.mpg.de})}
\and
Nick Vannieuwenhoven\thanks{KU Leuven, Department of Computer Science. (\email{nick.vannieuwenhoven@cs.kuleuven.be})}}

\begin{document}
\allowdisplaybreaks

\maketitle

\begin{abstract}
The canonical tensor rank approximation problem (TAP) consists of approximating a real-valued tensor by one of low canonical rank, which is a challenging non-linear, non-convex, constrained optimization problem, where the constraint set forms a non-smooth semi-algebraic set. We introduce a Riemannian Gauss--Newton method with trust region for solving small-scale, dense TAPs. The novelty of our approach is threefold. First, we parametrize the constraint set as
the Cartesian product of Segre manifolds, hereby formulating the TAP as a Riemannian optimization problem, and we argue why this parametrization is theoretically a good choice. Second, an original ST-HOSVD-based retraction operator is proposed. Third, we introduce a hot restart mechanism that efficiently detects when the optimization process is tending to an ill-conditioned tensor rank decomposition and which often yields a quick escape path from such spurious decompositions. Numerical experiments show improvements of up to three orders of magnitude in terms of the expected time to compute a successful solution over existing state-of-the-art methods.
\end{abstract}





\addtolength{\abovedisplayskip}{-2.5pt}
\addtolength{\belowdisplayskip}{-2.5pt}

\section{Introduction}\label{sec_introduction}
A \emph{simple} or \emph{rank-$1$ tensor} is the tensor product of vectors. Identifying tensors in coordinates with the $d$-arrays representing them with respect to some basis, the tensor product is given explicitly by the \emph{Segre map}, namely
\begin{align}\label{segre_map}
 \otimes: \R^{n_1} \times \R^{n_2} \times \cdots \times \R^{n_d} &\to \R^{n_1 \times n_2 \times \cdots \times n_d}\\ (\sten{a}{}{1}, \sten{a}{}{2}, \ldots, \sten{a}{}{d})
 &\mapsto
 \begin{bmatrix} a_{i_1}^{(1)} a_{i_2}^{(2)} \cdots a_{i_d}^{(d)} \end{bmatrix}_{i_1,i_2,\ldots,i_d=1}^{n_1,n_2,\ldots,n_d},\nonumber
\end{align}
where $\sten{a}{}{k} = [a_{i}^{(k)}]_{i=1}^{n_k}$. The image of this map comprises the simple tensors.
The \emph{tensor rank decomposition} or \emph{canonical polyadic decomposition} (CPD) that was proposed by Hitchcock \cite{Hitchcock1927} expresses $\tensor{A} \in \R^{n_1 \times n_2 \times \cdots \times n_d}$ as a linear combination of simple tensors:
\begin{equation}\tag{CPD}\label{CPD}
 \tensor{A} = \sum_{i=1}^r \sten{a}{i}{1} \otimes \sten{a}{i}{2} \otimes \cdots \otimes \sten{a}{i}{d}, \quad \sten{a}{i}{k} \in \R^{n_k}.
\end{equation}
The smallest $r$ such that a decomposition in \refeqn{CPD} is possible, is called the (canonical) \emph{rank} of $\tensor{A}$ and we denote it by $\rank(\tensor{A})$.

The set of all tensors of rank bounded by $r$ is then given by
\[
 \sigma_r := \{ \tensor{A} \in \R^{n_1 \times n_2 \times \cdots \times n_d}  \;|\; \rank(\tensor{A}) \le r \}.
\]
We say that $\sigma_r$ is \emph{non-defective} if on a dense subset of $\sigma_r$, in the Euclidean topology, a tensor only admits finitely many CPDs. \emph{Throughout this paper, we assume that $\sigma_r$ is non-defective.} This is not a serious limitation. Provided that $r < \frac{n_1 n_2 \cdots n_d}{n_1 + n_2 + \cdots + n_d - d + 1}$, an even stronger property holds \cite{COV2014,COV2017}: with few exceptions, there exists a Euclidean-dense subset $U$ of $\sigma_r$ such that every $\tensor{A} \in U$ is \emph{$r$-identifiable}, meaning that the summands in the factorization in \refeqn{CPD} are even \emph{uniquely determined}.

The CPD arises naturally in a myriad of applications; see, for example, those listed in \cite{Review2017}.
An extensive class where potentially high-rank CPDs are computed from dense tensors originates in machine learning and algebraic statistics, where the parameters of certain latent structure models admit a CPD \cite{GSS2005,SH2005,AMR2009,AGHKT2014}. Such models were recently surveyed in a unified tensor-based framework \cite{AGHKT2014}, including exchangeable single topic models, na\"ive Bayes models, hidden Markov models, and Gaussian mixture models.

Mostly, the tensors that one works with are only approximations of a theoretical low-rank tensor, as a consequence of measurement and representation errors or due to numerical computations. This is why the \emph{tensor rank approximation problem} (TAP), which consists of approximating a given tensor $\tensor{B} \in \R^{n_1 \times n_2 \times \cdots \times n_d}$ by a tensor $\tensor{A}$ of low canonical rank, i.e.,
\begin{equation}\tag{TAP}\label{TAP}
  \min_{\tensor{A} \in \sigma_r} \frac{1}{2} \| \tensor{A} - \tensor{B} \|^2_F,
\end{equation}
is the usual optimization problem that has to be solved. There are two standard assumptions that we make throughout this paper. First, we seek only one solution of this problem, rather than all solutions. Second, for some inputs $\tensor{B}$ this optimization problem is \textit{ill-posed} in the sense that only an \emph{infimum} exists. Unfortunately, the set of all such $\tensor{B}$ can have positive Lebesgue measure; see de Silva and Lim \cite{dSL2008}.

The setting that we are specifically interested in consists of \emph{dense} input tensors $\tensor{B}$ that can be well approximated by a tensor of small rank $r$ whose \emph{condition number} \cite{BV2017} (see \refsec{sec_condition_number}) is at most moderately large. In particular, we only consider the case of strictly subgeneric ranks $r < \frac{n_1 \cdots n_d}{n_1+\cdots+n_d-d+1}$.

\begin{remark}[Tucker compression] \label{rem_tucker_compression}
It is known that if a tensor $\tensor{B} \in \R^{n_1 \times \cdots \times n_d}$ of rank $r$ admits an orthogonal Tucker decomposition \cite{Tucker1966}
\[
\tensor{B} = (Q_1, \ldots, Q_d) \cdot \tensor{S} := \sum_{i_1=1}^{r_1} \cdots \sum_{i_d=1}^{r_d} s_{i_1, \ldots, i_d} \sten{q}{i_1}{1} \otimes \cdots \otimes \sten{q}{i_d}{d}
\]
with core tensor $\tensor{S} \in \R^{r_1 \times \cdots \times r_d}$ and orthonormal factor matrices $Q_k = [\sten{q}{i}{k}]_{i=1}^{r_k} \in \R^{n_k \times r_k}$, then one of $\tensor{B}$'s CPDs is $\tensor{B} = \sum_{i=1}^r (Q_1 \sten{s}{i}{1}) \otimes \cdots \otimes (Q_d \sten{s}{i}{d})$ where $\tensor{S} = \sum_{i=1}^r \sten{s}{i}{1} \otimes \cdots \otimes \sten{s}{i}{d}$ is a CPD of the core tensor $\tensor{S}$. The algorithm developed in this paper also applies when an orthogonal Tucker decomposition of $\tensor{B}$ is provided as input. In this case it suffices to compute a CPD of the (dense) core tensor $\tensor{S}$. This is useful when $r_i \ll n_i$ as it decreases the computational complexity dramatically; see \refapp{app_implementation}.

In practice, if $\tensor{B}$ can be well-approximated by a tensor $\tensor{A}$ of rank $r$, then a standard strategy for solving TAPs consists of first constructing an orthogonal Tucker decomposition of multilinear rank componentwise bounded by $(r, \ldots, r)$, and then solving the TAP for the core tensor. Since efficient algorithms for Tucker compression exist \cite{OST2008,DM2007,CC2010,Saibaba2016}, the key variable determining the computational cost for solving TAPs is typically the size of the approximation rank $r$ rather than the dimensions $n_i$.
\end{remark}

Several algorithms were proposed in the literature for solving \refeqn{TAP}. The state of the art can broadly be divided into two classes: alternating least squares (ALS) methods and general numerical optimization methods.
While in 2008 the review article \cite{KB2009} considered ALS (see \cite{Harshman1970,Caroll1970,CHLZ2012}) as ``the `workhorse' algorithm for~CP,'' we believe that nowadays quasi-Newton methods, such as those in \cite{Acar2011,Hayashi1982,Paatero1999,Paatero1997,Phan2013a,Sorber2013a,Tomasi2006,dSM2013}, have supplanted ALS-like strategies. For example, the default algorithm in the Tensorlab v3 software package \cite{Tensorlab} for finding an approximate CPDs is \texttt{nls\_gndl}, which is a Gauss--Newton (GN) method with trust region.
Furthermore, experiments in \cite{Acar2011,Sorber2013a,Phan2013a,Tomasi2006} have demonstrated that modern implementations of GN methods as described in \refsec{sec_trd} outperform standard ALS methods on all but the simplest of problems in terms of execution time and iteration count. For this reason, we do not consider ALS in the rest of this paper.

Instead, we adopt the \emph{Riemannian optimization} framework \cite{AMS2008} to solve \refeqn{TAP}. The set
$\sigma_r$ is known to be semi-algebraic \cite{dSL2008,BCR1998}, entailing that it is locally diffeomorphic to a Euclidean space at most, but not all, points.\footnote{More precisely, if $\sigma_r$ is of dimension $d$, then by ``most points'' we mean the locus $U \subset \sigma_r$ of \emph{$d$-dimensional smooth points} of $\sigma_r$: let $\sigma_r = \bigcup_{i=1}^k \Var{M}_i^{d_i}$ be a Nash stratification \cite[Proposition 2.9.10]{BCR1998}, where $\Var{M}_i^{d_i}$ is a $d_i$-dimensional Nash submanifold, and let $S = \{ i \;|\; d_i = d \}$. Then, $U := \bigcup_{j\in S} \Var{M}_j^d$, so that $U$ is locally (Nash) diffeomorphic to $\R^d$.}
Riemannian optimization methods \cite{AMS2008} exploit such a local smooth structure on the domain for effectively solving optimization problems. Such methods were already proposed for other tensor decompositions, like Tucker, hierarchical Tucker, and tensor train decompositions; see respectively \cite{SL2010,Ishteva2011}, \cite{UV2013,dSH2015}, and \cite{HRS2012,Steinlechner2016,KSV2014}. Hitherto, a Riemannian optimization method is lacking for the TAP. This paper develops such a method.

\subsection{Contributions}
The main contribution of this paper is an efficient Riemannian trust region (RTR) method for finding well-conditioned \cite{V2017,BV2017} solutions of \refeqn{TAP}. The numerical experiments in \refsec{sec_numerical_experiments} show that this method can outperform state-of-the-art classic optimization methods on the key performance criterion, namely the expected time to success, by up to three orders of magnitude.

The design of our method was guided by a geometrical analysis of the TAP; see \cref{sec_tap_analysis,sec_ill_H_p}.
We argue that none of the state-of-the-art general optimization methods that were adapted to the TAP exploit the essential geometry of the optimization domain (see \refsec{sec_trd}).
This analysis yields three main insights: first, it suggests that the parametrization of the domain of the optimization problem should be \emph{homothetical}; second, parameterizing $\sigma_r$ via the manifold of rank-one tensors is probably among the best choices and allows us to formulate the TAP as a Riemannian optimization problem; and, third, certain CPDs are ill-behaved for (Riemannian) Gauss--Newton methods, hence we argue that they should be avoided, e.g., via \emph{hot restarts} (see \refsec{sec_hot_restarts}). This last point motivates the name for our method: Riemannian Gauss--Newton method with hot restarts, or RGN-HR for short.

Unfortunately, the main strength of RGN-HR is also its significant limitation: by design it can find only well-conditioned optimizers of \refeqn{TAP}. In particular, our method can only find \emph{isolated} solutions. We believe that this is not a severe limitation in practice because in many applications only (locally) unique rank-one summands are of interest. We say that non-isolated solutions are \emph{ill-posed} or, equivalently, that their condition number is unbounded; {see} \refprop{prop_curve_is_illcond}.

\subsection{Outline}
The rest of the paper is structured as follows. In the next section, we recall the condition number of the CPD.
In \refsec{sec_trd}, we analyze the geometry of the \reftag{TAP}, argue that this structure has not been exploited hitherto, and explain how it can be exploited in a Riemannian optimization method. \Cref{sec_RGN} describes the outline of RGN-HR. The strategy {for dealing with ill-conditioned Hessian approximations} is described in greater detail in \cref{sec_hot_restarts}. Riemannian optimization methods require the specification of a retraction operator. We explain our choice, a product ST-HOSVD retraction, in \cref{sec_retraction}. Numerical experiments demonstrating the efficacy of RGN-HR are featured in \refsec{sec_numerical_experiments}, and \cref{sec_application} illustrates the application of our method on a tensor originating in fluorescence spectroscopy. The final section presents our conclusions and outlook.

\subsection{Notation}
For ease of reading, vectors are typeset in lower-case boldface letters~($\vect{x}$); matrices in upper-case letters ($A$); tensors in upper-case calligraphic letters ($\tensor{A}, \tensor{B}$); and varieties and manifolds in an alternative calligraphic font ($\Var{S}, \Var{M}$).

The \emph{vectorization} of a tensor $\tensor{A} = [a_{i_1,\ldots,i_d}] \in \R^{n_1 \times \cdots \times n_d}$ is defined in the usual way as $\operatorname{vec}(\tensor{A}) := \left[\begin{smallmatrix} a_{1,\ldots,1} & a_{2,1,\ldots,1} & \cdots & a_{n_1,\ldots,n_{d-1},n_d-1} & a_{n_1,\ldots,n_d} \end{smallmatrix}\right]^T$.
A real vector space $\R^n$ is endowed with the Euclidean inner product $\langle \vect{x}, \vect{y} \rangle = \vect{x}^T \vect{y}$. This induces the Euclidean norm $\lVert \vect{x} \rVert := \sqrt{\langle \vect{x}, \vect{x} \rangle}$.
The corresponding \emph{spectral norm} of a matrix $M\in\R^{m\times n}$ is denoted by $\| M \|_2:=\max_{\vect{x}\in\R^n, \lVert \vect{x} \rVert=1} \lVert M \vect{x} \rVert.$ {The $n$-dimensional unit sphere in $\R^{n+1}$ is denoted by $\mathbb{S}^n$.} We make $\R^{n_1 \times n_2 \times \cdots \times n_d}$ a normed vector space by defining the norm of a tensor $\tensor{A}$ to be the \emph{Frobenius norm} $\lVert \tensor{A}\rVert_F := \|\operatorname{vec}(\tensor{A})\|$.

The Moore--Penrose \emph{pseudoinverse} of $M$ \cite[chapter 3, section 1]{MPT} is denoted by $M^\dagger$. We denote by $\varsigma_{\max}(M), \varsigma_{\min}(M)$, and $\varsigma_k(M)$, the largest, smallest, and $k$th largest singular value, respectively, of the linear operator or matrix $M$.

The set of rank-$1$ tensors in $\R^{n_1 \times n_2 \times \cdots \times n_d}$ is a smooth manifold called the \emph{Segre manifold}\footnote{It is standard to call the projective variety of which $\Var{S}$ is the cone the \emph{Segre variety}. To avoid confusion and for brevity, we prefer ``Segre manifold'' to ``affine cone over the Segre variety.''} \cite[section 4.3.5]{Landsberg2012}, which we denote by $\Var{S}_{n_1,\ldots,n_d}$. Moreover, we define
\begin{equation*} 
\Sigma = \sum_{k=1}^d (n_k - 1)\quad\text{and}\quad \Pi = \prod_{k=1}^d n_k,
\end{equation*}
such that $\dim(\Var{S}_{n_1,\ldots,n_d})=\Sigma+1$ and $\dim (\R^{n_1 \times n_2 \times \cdots \times n_d}) = \Pi$. We often abbreviate $\Var{S}:=\Var{S}_{n_1,\ldots,n_d}$ when the tuple of integers $(n_1,\ldots,n_d)$ is clear from the context.

Let $\Var{M}$ be a manifold and $x \in \Var{M}$. The \emph{tangent space} to $\Var{M}$ at $x$ is denoted by $\Tang{x}{\Var{M}}$, and the \emph{tangent bundle} of $\Var{M}$ is denoted by $\Var{T}\Var{M}$; see Lee \cite{Lee2013}.

\section{The condition number of CPDs} \label{sec_condition_number}
We recall the condition number of the CPD from \cite{BV2017}. It plays a pivotal role in the analyses in \cref{sec_trd,sec_hot_restarts} which led to the main insights for improving the state-of-the-art optimization methods.

The \emph{Segre manifold} in $\R^{n_1  \times \cdots \times n_d}$ is denoted by $\Var{S}:=\Var{S}_{n_1,\ldots,n_d}$. It can be obtained as the image of the Segre map, as defined in \refeqn{segre_map}, after removing the zero tensor:
$$
\Var{S} = \{ \sten{a}{ }{1} \otimes \cdots \otimes \sten{a}{ }{d} \mid \sten{a}{ }{k}\in\R^{n_k} \} \setminus \{0\}.
$$
The set of tensors of rank at most $r$ is then defined as the image of the \emph{addition map}
\begin{equation}\label{Phi}
\Phi: \Var{S}^{\times r}:= \Var{S}\times \cdots \times \Var{S} \to \R^{n_1\times \cdots \times n_d},\;
(\tensor{B}_1, \ldots, \tensor{B}_r) \mapsto \tensor{B}_1 + \cdots + \tensor{B}_r.
\end{equation}
In \cref{sec_introduction}, we denoted it by $\sigma_r$, but for emphasizing the dependence on~$\Var{S}$ we will write
$
\sigma_r(\Var{S}):=\Phi(\Var{S}^{\times r}).
$
The derivative of $\Phi$ at $\tuple{p}=(\tensor{B}_1,\ldots,\tensor{B}_r)\in \Var{S}^{\times r}$ is
\[
\deriv{\tuple{p}}{\Phi} : \Tang{\tensor{B}_1}{\Var{S}} \times \cdots \times \Tang{\tensor{B}_r}{\Var{S}} \to \Tang{\Phi(\tuple{p})}{\R^\Pi}, \; (\dot{\vect{p}}_1, \ldots, \dot{\vect{p}}_r ) \mapsto \dot{\vect{p}}_1 + \cdots + \dot{\vect{p}}_r.
\]

The \emph{(local) condition number} \cite{BC2013,Higham1996} of the CPD at $\tuple{p}$ is
\begin{equation*}
\kappa (\tuple{p}, \tensor{A}) := \lim_{\epsilon \to 0} \; \max_{\tensor{A}' \in (B_{\epsilon}(\tensor{A}) \cap \sigma_r(\Var{S}))} \frac{\| \Phi_{\tuple{p}}^{-1}(\tensor{A}) - \Phi_{\tuple{p}}^{-1}(\tensor{A}') \|_F}{\| \tensor{A} - \tensor{A}' \|_F},
\end{equation*}
where $B_\epsilon(\tensor{A})$ is an $\epsilon$-ball centered at $\tensor{A}$, if there is a local inverse function $\Phi_{\tuple{p}}^{-1}$ of $\Phi$ at $\tuple{p}$; otherwise, the local condition number is defined as $\kappa (\tuple{p}, \tensor{A}):=\infty$. The norm in the numerator of the above definition is the Euclidean norm on $(\R^{n_1 \times \cdots \times n_d})^{\times r} \simeq \R^{\Pi \times r}$; that is, if $\tuple{p}' =(\tensor{B}_1',\ldots,\tensor{B}_r') \in \Var{S}^{\times r}$, then
\[
 \| \tuple{p} - \tuple{p}' \|_{F} := \bigl\| \begin{bmatrix} \operatorname{vec}(\tensor{B}_1 - \tensor{B}_1') & \cdots & \operatorname{vec}(\tensor{B}_r - \tensor{B}_r') \end{bmatrix} \bigr\|_F.
\]%
Note that $\kappa (\tuple{p}, \tensor{A})$ is completely determined by the choice of $\tuple{p}$. We therefore set $\kappa(\tuple{p}):=\kappa (\tuple{p}, \tensor{A})$ and call it the \emph{condition number of the decomposition $\tuple{p}$}.
In \cite{BV2017} we showed that it is the inverse of the smallest singular value of $\deriv{\tuple{p}}{\Phi}$:
\begin{equation}\label{kappa}
\kappa(\tuple{p})=\frac{1}{\varsigma_{n}(\deriv{\tuple{p}}{\Phi})}, {\quad\text{where } n := r \cdot \dim \Var{S} = r (\Sigma+1).}
\end{equation}
By the assumption that $\sigma_r$ is non-defective we have $\dim \Var S^{\times r} = \dim \sigma_r(\Var S)$. Hence, the condition number is finite on a dense subset of $\Var S^{\times r}$; this follows from \cite[proof of Theorem 1.1]{BV2017}. If $\tensor{A}\in\sigma_r$ is $r$-identifiable then all $\tuple{p}\in\Phi^{-1}(\tensor{A})$ admit the same condition number \cite{BV2017}. In this case it makes sense to speak about the \emph{condition number of $\tensor{A}$}.

Recall that the condition number describes the first-order behavior of the local inverse of~$\Phi$ at $\tuple{p}$, if it exists. It thus measures the \emph{sensitivity} of $\tuple{p}$ with respect to perturbations of $\tensor{A}=\Phi(\tuple{p})$ within $\sigma_r(\Var S)$.
A large condition number means that a small perturbation $\tensor{A}' \approx \tensor{A}$ can result in a {strongly perturbed} CPD $\tuple{p}' = (\tensor{B}_1',\ldots,\tensor{B}_r') \in \Var{S}^{\times r}$. This is because we have the asymptotically sharp bound
\(
\| \tuple{p}' - \tuple{p}\|_{F} \lesssim \kappa(\tuple{p}') \cdot \| \tensor{A}' - \tensor{A} \|_F.
\)
What about small condition numbers? Recall from \cite{BV2017} that the relative error between $\tensor{B}_i'$ in $\tuple{p}'$ and $\tensor{B}_i$ in $\tuple{p} = (\tensor{B}_1,\ldots,\tensor{B}_r) \in \Var{S}^{\times r}$ is asymptotically bounded by
\begin{equation*} 
 \frac{\|\tensor{B}_i' - \tensor{B}_i\|_F}{\|\tensor{B}_i'\|_F} \lesssim \kappa_i^{\text{rel}}(\tuple{p}') \frac{\|\tensor{A}'-\tensor{A}\|_F}{\|\tensor{A}'\|_F}, \text{ where }
 \kappa_i^{\text{rel}}(\tuple{p}') := \kappa(\tuple{p}') \frac{\|\tensor{A}'\|_F}{\|\tensor{B}_i'\|_F}.
\end{equation*}
If $\kappa(\tuple{p}') \|\tensor{A}'-\tensor{A}\|_F > \|\tensor{B}_i'\|_F$, then this ``upper bound'' is larger than~$1$; hence, interpreting $\tensor{B}_i'$ in an application is nonsensical without further analysis. Note that \emph{a CPD with a small condition number $\kappa(\tuple{p}')$ can contain rank-$1$ terms $\tensor{B}_i'$ that are well-conditioned}, i.e., $\|\tensor{B}_i'\|_F \approx \|\tensor{A}'\|_F$, \emph{and others that are ill-conditioned}, i.e., $\|\tensor{B}_i'\|_F \ll \|\tensor{A}'\|_F$.

For the above reasons, in applications, where one naturally faces approximate data, a small condition number is required {to ensure that at least some rank-$1$ tensors in the computed CPD $\tuple{p}'$ are close to the corresponding terms in the} true CPD $\tuple{p}$.
Following \cite{BV2017}, we say that the TAP is \emph{ill-posed} at the CPD $\tuple{p} \in \Var{S}^{\times r}$ if $\kappa(\tuple{p}) = \infty$.
\emph{In this paper, we deal exclusively with CPDs that are not ill-posed solutions of \refeqn{TAP}.} That is, we always assume that a (local) optimizer $\tuple{p}$ of \refeqn{TAP} has $\kappa(\tuple{p}) < \infty$. To the best of our knowledge, RGN-HR is the first method that exploits this property.

\section{Analysis of the TAP} \label{sec_trd}
In this section we explain how the approach for solving \refeqn{TAP} in this article differs from existing methods. First, we summarize the state of the art, then recall the Riemannian optimization framework from \cite{AMS2008} and propose our Riemannian formulation of the TAP, next consider the geometry behind the TAP to motivate theoretically why we think that our formulation is preferable, and finally argue that the local convergence rate estimates of the Riemannian formulation are superior to those of state-of-the-art optimization methods in certain cases.

\subsection{State of the art} \label{sec_stateoftheart}
All of the optimization methods that we know of, i.e., \cite{Acar2011,Harshman1970,Caroll1970,Hayashi1982,Paatero1999,Paatero1997,Phan2013,Phan2013a,Sorber2013a,Tomasi2006,DeSterck2012,dSM2013,CHLZ2012,LUZ2015}, choose a standard parameterization of tensors of rank bounded by $r$, called \emph{factor matrices}. That is, a tensor $\tensor{A}$ is represented as
\begin{align} \label{factor_matrices_def}
 \tensor{A} := {\text\textlbrackdbl} A_1, A_2, \ldots, A_d {\text\textrbrackdbl} := \sum_{i=1}^r \sten{a}{i}{1}\otimes\sten{a}{i}{2}\otimes\cdots\otimes \sten{a}{i}{d},
\end{align}
where $A_k := [\sten{a}{i}{k}]_{i=1}^r \in \R^{n_k\times r}$.
The TAP is then formulated as the following unconstrained optimization problem over $D := \R^{n_1 \times r} \times \cdots \times \R^{n_d \times r}$,
\begin{align}\label{eqn_optimization_problem_classic}
\min_{(A_1,\ldots,A_r)\in D} \,\frac{1}{2} \bigl\| {\text\textlbrackdbl} A_1, \ldots, A_d {\text\textrbrackdbl} - \tensor{B} \bigr\|_F^2,
\end{align}
which can be handled with traditional methods. In the remainder, we restrict our discussion to the widespread, effective GN methods for this problem, such as those proposed in  \cite{Acar2011,Hayashi1982,Paatero1999,Paatero1997,Phan2013a,Sorber2013a,Tomasi2006}.

A key step in GN methods is (approximately) solving the normal equations
\(
 (J^T J) \vect{p} = -J^T \vect{r},
\)
where $\vect{r} = \operatorname{vec}\bigl( {\text\textlbrackdbl} A_1, A_2, \ldots, A_d {\text\textrbrackdbl} - \tensor{B} \bigr)$ is the residual vector and the Jacobian matrix $J \in \R^{\Pi \times r(\Sigma+d)}$ has the following well-known block structure:
\begin{align}\label{eqn_overparam_jacobian}
J =
\begin{bmatrix}
[I_{n_1} \otimes \sten{a}{i}{2} \otimes \cdots \otimes \sten{a}{i}{d}]_{i=1}^r & \cdots & [\sten{a}{i}{1} \otimes \cdots \otimes \sten{a}{i}{d-1} \otimes I_{n_d}]_{i=1}^r
\end{bmatrix}.
\end{align}
The matrix $J^T J$ is the GN approximation of the Hessian of the objective function in \refeqn{eqn_optimization_problem_classic}. Given the (approximate) solution $\vect{p}$ of the normal equations, the next iterate is then determined either by a line or plane search, or via the trust region framework.

A complication is that the Gram matrix $J^TJ$ is never of full rank for geometrical reasons, so that the normal equations are ill-posed; see \cite[section 2]{V2017} and \refsec{sec_tap_analysis} below. The aforementioned methods differ in how they resolve this problem; for example, it can be solved by adding a regularization term, via a truncated conjugate gradient (CG) method, or by computing the Moore--Penrose pseudoinverse.

\begin{remark}\label{rem_numerical_accuracy}
For reasons of numerical accuracy it is in general not recommended to solve the normal equations; rather one usually solves the equivalent least-squares problem $J \vect{p} = -\vect{r}$. However, in the case of the TAP, both the Jacobian matrix $J$ and its Gram matrix $J^T J$ are very structured. The usual $QR$-factorization in a least-squares solver destroys this structure, resulting in significantly slower solves and excessive memory consumption. For this reason, the references \cite{Acar2011,Hayashi1982,Paatero1999,Paatero1997,Phan2013a,Sorber2013a,Tomasi2006} prefer solving the normal equations rather than the least-squares problem.
\end{remark}

\subsection{A Riemannian formulation} \label{sec_riemannian_formulation}
Classic optimization methods can be generalized to the framework of \textit{Riemannian optimization} \cite{AMS2008} for solving optimization problems where the constraint set $\Var{M}$ is a \emph{Riemannian manifold} \cite[chapter 13]{Lee2013}. A particular instance of a Riemannian optimization problem that is of relevance to this paper is the least-squares optimization problem
\begin{equation}\label{gauss-newton-least-squares}
  \min_{x \in \Var{M}}
 \,\frac{1}{2} \Vert F(x)\Vert^2,
\end{equation}
where $F:\Var{M}\to \R^M$ is a smooth {objective} function with $M\geq \dim \Var{M}$. Riemannian optimization methods are defined with respect to a \emph{Riemannian metric} \cite[chapter 13]{Lee2013} and a \emph{retraction operator} $R$ on $\Var{M}$; see \cite[section 4.1]{AMS2008}, and \refsec{sec_retraction} below. Since we will assume that $\Var{M}$ is embedded in $\R^N$, we take the metric induced by this ambient Euclidean space. The classic GN method generalizes to functions on manifolds; the outline of this Riemannian Gauss--Newton (RGN) method for solving \refeqn{gauss-newton-least-squares} from \cite[Algorithm 14]{AMS2008} is recalled as \refalg{RGN-framework}.

\begin{algorithm2e}[tb]\footnotesize
\SetAlgoLined
\KwData{Riemannian manifold $\Var M$; retraction $R$ on $\Var M$; function $F:\Var M \to \Var \R^M$.}
\KwIn{A starting point $x_0\in\Var{M}$.}
\KwOut{A sequence of iterates $x_k$.}
\For{$k = 0, 1, 2, \ldots$}{
Solve the GN equation $\bigl( (\deriv{x_k}{F})^* \circ (\deriv{x_k}{F}) \bigr)\;\eta_k = - (\deriv{x_k}{F})^* \bigl( F(x_k) \bigr)$
for the unknown $\eta_k\in\Tang{x_k}{\Var{M}}$\;
Set $x_{k+1} \leftarrow R_{x_k}(\eta_k)$\;
}
\caption{Riemannian Gauss-Newton method}
\label{RGN-framework}
\end{algorithm2e}

One might hope that the TAP could be formulated as in \refeqn{gauss-newton-least-squares} with domain $\Sec{r}{\Var{S}}$; alas, this set is \emph{not} a manifold. However, the product $\Var{S}^{\times r}$ is a manifold.
We propose parameterizing $\sigma_r(\Var{S})$ via the addition map~$\Phi$, so that we can formulate the TAP as
\begin{equation}\tag{TAP*}\label{TAP3}
 \min_{\tuple{p}\in\Var{S}^{\times r}} \frac{1}{2} \| \Phi(\tuple{p}) - \tensor{B} \|_F^2.
\end{equation}
Assuming that \refeqn{TAP} is well posed, it is clear that $\tuple{p} \in \Var{S}^{\times r}$ is a global minimizer of \refeqn{TAP3} if and only if $\Phi(\tuple{p})$ is a global minimizer of \refeqn{TAP}.
It is also evident that every $\tuple{p}$ such that $\Phi(\tuple{p})$ is a local minimizer of \refeqn{TAP} is a local minimizer of \refeqn{TAP3}. However, while we have no explicit examples, we expect based on considerations about $r$-identifiability that the converse is false at least for some $\tensor{B} \in \R^{n_1 \times \cdots \times n_d}$; that is, the formulation \refeqn{TAP3} may introduce additional local minima.

\subsection{Analysis}\label{sec_tap_analysis}
The formulation as Riemannian optimization problem in \refeqn{TAP3} distinguishes the proposed approach from the state-of-the-art methods described in \cref{sec_stateoftheart} which solve \refeqn{eqn_optimization_problem_classic} instead.
In this subsection, we interpret both approaches as special cases of a family of possible Riemannian formulations of the TAP. We argue that in this family \emph{the formulation \refeqn{TAP3} is theoretically a good choice when a RGN method is utilized for solving it}, the condition number of the obtained solution is small, and orthonormal bases are selected for the tangent spaces.

The parameterizations of $\sigma_r(\Var{S})$ via factor matrices in $\R^{n_1 \times r} \times \cdots \times \R^{n_d \times r}$ and via rank-$1$ tensors in $\Var{S}^{\times r}$ are both special cases of the following scenario. Let $\Var E$ be a Riemannian manifold with $\dim \Var{E} \geq \dim \Var{S}^{\times r}$, and let
\(
\Psi: \Var E \rightarrow \sigma_r(\Var{S})
\)
be a smooth, surjective parametrization of $\sigma_r(\Var{S})$ with corresponding optimization problem
\begin{equation}\label{optimization_problem_classic}
  \min_{x \in \Var{E}}\,\frac{1}{2} \|\Psi(x) - \tensor{B} \|_F^2.
\end{equation}
We furthermore assume that there is a smooth map $\pi:\Var{E} \to \Var{S}^{\times r}$ so that the diagram
$$
\xymatrix{
\Var E  \ar[dr]_\Psi \ar[r]^\pi&\Var{S}^{\times r}\ar[d]^\Phi\\
 &\sigma_r(\Var{S})}
$$
commutes, where $\Phi$ is the addition map from \refeqn{Phi}. We believe that a meaningful parametrization of $\sigma_r(\Var{S})$ should supply a method $\pi$ for obtaining the collection of rank-$1$ summands in \refeqn{CPD}, as they are often of interest in applications. This explains why the assumption on the existence of $\pi$ is justified.

In the case of factor matrices, $\pi$ is {the columnwise Khatri--Rao product $\odot$, i.e., $\pi(A_1,\ldots,A_d) = A_1 \odot \cdots \odot A_d := [\sten{a}{i}{1}\otimes \cdots \otimes \sten{a}{i}{d}]_{i=1}^r$,} and $\Psi={\text\textlbrackdbl}\cdot {\text\textrbrackdbl}$ is the map from \refeqn{factor_matrices_def}. On the other hand, in our approach, $\pi$ is the identity map and $\Psi=\Phi$. Both the factor matrices and the parametrization over $\Var{S}^{\times r}$ may serve as guiding examples, but in the following discussion $(\Var E,\Psi)$ can be any suitable parametrization of $\sigma_r(\Var S)$.

Let $F(x)=\Psi(x) - \tensor{B}.$
After fixing a choice of orthonormal bases for $\Tang{x}{\Var{E}}$ and $\Tang{\Psi(x)}{\R^{n_1 \times \cdots \times n_d}}$, the RGN method applied to optimization problem \refeqn{optimization_problem_classic} solves, in coordinates, either the normal equations or a least squares problem, i.e., either
\begin{equation*} 
 (J^TJ) \eta = -J^T (\Psi(x) - \tensor{B})
 \;\text{ or }\;
 J \eta = - (\Psi(x) - \tensor{B}),
\end{equation*}
respectively, where {$J$ is the matrix of} $\deriv{x}{F}=\deriv{x}{\Psi}$ with respect to the chosen bases. The sensitivity of both computational problems depends {critically} on the \emph{effective condition number} of $J$, namely $\kappa_2(J) := \Vert J^\dagger \Vert_2\,\Vert J\Vert_2$; see, e.g., \cite[section 5.3.8]{matrix_computations}.

We now show that some parameterizations are inherently worse than others in the sense that $\kappa_2(J)$ will be large. To this end, we require the mild technical assumption that $\pi$ is a surjective \emph{submersion}: $\deriv{x}{\pi}$ is of maximal rank $n := \dim \Var{S}^{\times r}$ for all $x \in \Var{E}$. One verifies that the identity and Segre maps are surjective submersions.

Let $x\in \Var{E}$ and $\tuple p:=\pi(x)\in\Var{S}^{\times r}$.
Choose any orthonormal basis for $\Tang{\tuple{p}}{\Var{S}^{\times r}}$, and let~$P$, respectively $Q$,
be the matrix of $\deriv{x}{\pi}$, respectively $\deriv{\tuple{p}}{\Phi}$, with respect to the chosen bases on the three tangent spaces. Let $m := \dim \Var{E} \ge \dim \Var{S}^{\times r} = n$, then we have $J \in \R^{\Pi \times m}$, $P \in \R^{n \times m}$, and $Q \in \R^{\Pi \times n}$.
From the chain rule and the commuting diagram above it follows that $J = Q P$.
As in~\refeqn{kappa}, let $\kappa(\tuple p)$ denote the condition number of the CPD at $\tuple p$. Recall that the singular values of a linear operator are the singular values of any of its matrix representations with respect to orthonormal bases on the domain and image. Hence,
\(
 \kappa(\tuple{p}) = \bigl(\varsigma_n ( \deriv{\tuple{p}}{\Phi} ) \bigr)^{-1} = \bigl(\varsigma_n ( Q ) \bigr)^{-1}.
\)
As explained before, we assume that $\kappa(\tuple p) <\infty$,
so that it follows that $Q$ is injective and thus has a left inverse $Q^\dagger$.
Moreover, $P$ has full row rank because we assumed $\deriv{x}{\pi}$ is a submersion, so that $P$ has a right inverse $P^\dagger$. We find $JP^\dagger = Q$ with pseudoinverse $PJ^\dagger = Q^\dagger$. On the other hand, we also have $Q^\dagger J = P$ and $J^\dagger Q = P^\dagger$. From the submultiplicativity of the spectral norm, we then derive
$\Vert Q\Vert \leq \Vert J\Vert\Vert P^\dagger\Vert$,
$\Vert Q^\dagger\Vert \leq \Vert P\Vert\Vert J^\dagger\Vert$,
$\| P \| \le \| J \| \|Q^\dagger\|$, and
$\| P^\dagger \| \le \|J^\dagger\| \|Q\|$.
Combining the first and second bounds, and the third and fourth bounds yields, respectively,
\begin{align} \label{eqn_lower_bound}
\kappa_2(Q) \le \kappa_2(J) \kappa_2( P )
\quad\text{and}\quad
\kappa_2( P ) \le \kappa_2(J) \kappa_2(Q).
\end{align}
We can write these condition numbers in terms of the singular values of the derivative operators, namely
\begin{align} \label{eqn_kappa_intrinsic}
\kappa_2(P) = \frac{\varsigma_1(\deriv{x}{\pi})}{\varsigma_n(\deriv{x}{\pi})},
\quad
\kappa_2(Q) = \frac{\varsigma_1(\deriv{\tuple{p}}{\Phi})}{\varsigma_n(\deriv{\tuple{p}}{\Phi})}
\;\text{ and }\;
\kappa_2(J) = \frac{\varsigma_1(\deriv{x}{\Psi})}{\varsigma_n(\deriv{x}{\Psi})},
\end{align}
which are intrinsic expressions that do not depend on the choice of bases.
For proceeding, we need the following lemma, which is proved in \refapp{app_proof_sthosvd_retraction}.
\begin{lemma}\label{norm_of_phi} The spectral norm of $\deriv{\tuple p}{\Phi}$ satisfies
$1\leq \Vert \deriv{\tuple p}{\Phi}\Vert \le \sqrt{r}$.
\end{lemma}

Since $\|\deriv{\tuple{p}}{\Phi}\| = \varsigma_1( \deriv{\tuple{p}}{\Phi} )$ and $\varsigma_n( \deriv{\tuple{p}}{\Phi})  = \kappa(\tuple{p})^{-1}$, we derive from plugging \refeqn{eqn_kappa_intrinsic} into \refeqn{eqn_lower_bound}, and using \reflem{norm_of_phi} that
\begin{equation}\label{isometry_ineq2}
\max\left\{
\kappa(\tuple{p}) \cdot \frac{\varsigma_{n}(\deriv{x}{\pi})}{\varsigma_{1}(\deriv{x}{\pi})},\;
\frac{1}{\sqrt{r}} \cdot \left( \kappa(\tuple{p})\, \frac{\varsigma_{n}(\deriv{x}{\pi})}{\varsigma_{1}(\deriv{x}{\pi})} \right)^{-1}
\right\}
\leq \kappa_2(J).
\end{equation}
Note that the condition number $\kappa(\tuple p)$ should be interpreted as a constant here, because it is an inherent property of the decomposition $\tuple p$ that does not depend on the parametrization we choose for $\Var{S}^{\times r}$.

Based on the above lower bound, a natural idea is constructing a parameterization $\Var{E}$ with maps $\Psi$ and $\pi$ in such a way that the lower bound in \refeqn{isometry_ineq2} is minimized. However, we do not know how it should be constructed, if it exists at all. Therefore, we settle for a good choice instead. Recall from the previous section that decompositions $\tuple{p}$ with large condition number are often not of interest in applications, because the components cannot be interpreted rigorously due to their sensitivity to noise. We thus focus on those cases with a small $\kappa(\tuple{p})$, in which case it is clear that a small lower bound in \refeqn{isometry_ineq2} is obtained by choosing a parametrization $\pi$ with likewise small $\kappa_2(P)$. The minimal $\kappa_2(P) = 1$ and corresponds to a {\emph{homothetical}} parametrization $\pi : \Var{E} \to \Var{S}^{\times r}$; {that is,} $\varsigma_{1}(\deriv{x}{\pi})=\varsigma_{n}(\deriv{x}{\pi})$ for all $x\in\Var E$.
Based on this analysis, we propose the choice $\pi=\mathrm{id}$, yielding the lower bound $\kappa(\tuple{p}) \le \kappa_2(J)$.%

An example that may lead to a large lower bound in \refeqn{isometry_ineq2} is the parameterization via factor matrices with $\pi$ equal to the $r$-fold Segre map. In this case, the parameter space is $\Var{E} = \R^{n_1 \times r} \times \cdots \times \R^{n_d \times r}$. Assume that $(A_1, \ldots, A_d) \in \Var{E}$. Then, $\deriv{x}{\pi}$ with respect to the standard basis on $\Tang{x}{\Var{E}} \simeq \R^{n_1 r + \cdots + n_d r}$ and the standard basis on $\Tang{\pi(x)}{(\R^{\Pi})^{\times r}} \simeq \R^{r \Pi}$ is
\[
 T := \operatorname{diag}( T_1, \ldots, T_r ), \text{ where } T_i = \begin{bmatrix} I \otimes \sten{a}{i}{2} \otimes \cdots \otimes \sten{a}{i}{d} & \cdots & \sten{a}{i}{1} \otimes \cdots \otimes \sten{a}{i}{d-1} \otimes I \end{bmatrix}
\]
and $A_k = [\sten{a}{i}{k}]_{i=1}^r$. The matrix $T_i$ is what \cite{V2017} calls a \emph{Terracini matrix} corresponding to a rank-$1$ tensor. Its spectrum was completely characterized in \cite[corollary 18]{V2017} in the \emph{norm-balanced case}, i.e., the case whereby $\alpha_i := \|\sten{a}{i}{1}\| = \cdots = \|\sten{a}{i}{d}\|$. It follows from these results that $\varsigma_1(T_i) = \sqrt{d} \alpha_i^{d-1}$ and the smallest nonzero singular value is $\alpha_i^{d-1}$. From the block diagonality of $T$, it follows that $\frac{\varsigma_{1}(\deriv{x}{\pi})}{\varsigma_{n}(\deriv{x}{\pi})} = \sqrt{d} \bigl( \frac{\max_i \alpha_i}{\min_i \alpha_i} \bigr)^{d-1}$, which is large for all tensors with large differences in norm between the rank-$1$ tensors, yielding the potentially very large lower bound
\(
 \frac{\sqrt{d}}{\sqrt{r} \, \kappa(\tuple{p})} \bigl( \frac{\max_i \alpha_i}{\min_i \alpha_i} \bigr)^{d-1} \le \kappa_2( J ).
\)

\subsection{Convergence} \label{sec_convergence}
We showed in \cite{BV2017bis} that the multiplicative constant in the local linear rate of convergence of a RGN method near a non-exact local optimizer~$\tuple{p}^*$, i.e., $f(\tuple{p}^*)\ne0$, is a multiple of the square of the condition number $\kappa(\tuple{p}^*)$ in \refeqn{kappa}.

Likewise, it follows from the proof\footnote{Specifically the penultimate step in the penultimate formulae on page 208.} of Theorem 7 in \cite{DK2002} that the multiplicative constant in the local linear rate of convergence of GN methods to a non-exact solution~$x^*$ of the objective function $f(x) = \frac{1}{2}\|g(x)\|^2$ whose Jacobian $J_g$ has constant rank on an open neighborhood containing $x^*$ is a multiple of $\|J_g^\dagger\|_2^2$. Consequently, the local convergence behavior of the GN methods mentioned in \refsec{sec_stateoftheart} applied to the TAP is governed by $\|J^\dagger\|_2^2$ with $J$ as in \refeqn{eqn_overparam_jacobian}. In \cite{V2017} it was argued that $\kappa(\tuple{p}) := \|J^\dagger\|_2$ is naturally the condition number of the TAP in \refeqn{eqn_optimization_problem_classic} at the CPD $\tuple{p}$.

As we have shown at the end of the previous subsection, $\|J^\dagger\|_2$ with $J$ as in \refeqn{eqn_overparam_jacobian} will be large at $\tuple{p} = (\tensor{B}_1,\ldots,\tensor{B}_r)$ if there are large differences in the scaling factors $\alpha_i := \|\tensor{B}_i\|_F$ of the rank-$1$ terms in $\tuple{p}$. The convergence of GN methods will, at least theoretically, deteriorate with increasing $\frac{\max_i \alpha_i}{\min_i \alpha_i}$. However, since we proved in \cite[Proposition 5.1]{BV2017} that the condition number in \refeqn{kappa} is \emph{independent of the scaling of the rank-$1$ terms} appearing in the CPD $\tuple{p}$, i.e., $\kappa((t_1 \tensor{B}_1, \ldots, t_r \tensor{B}_r)) = \kappa(\tuple{p})$ for all $t_i\in\R\backslash\{0\}$, it follows that, theoretically, the local rate of convergence of RGN methods for \refeqn{TAP3} is not affected by large differences in scaling. \emph{These methods might thus significantly outperform the GN methods for \refeqn{eqn_optimization_problem_classic} for CPDs with large differences in scale}. This behavior is confirmed experimentally in \refsec{sec_numerical_experiments}.

\section{The proposed Riemannian trust region method}\label{sec_RGN}
The considerations of the preceding section and because we are not aware of any evident homothetical parametrization of $\Var{S}^{\times r}$ other than the trivial one, i.e., $\pi=\mathrm{id}$, motivated us to avoid optimizing over a flat parameter space. Instead we will optimize over the Riemannian manifold $\Var{S}^{\times r}$. To this end, let us write
$$
 f(\tuple{p}) = \frac{1}{2} \|F(\tuple{p})\|^2  \;\text{ with }
 F : \Var{S}^{\times r} \to \R^\Pi,\; \tuple{p} \mapsto \operatorname{vec}( \Phi(\tuple{p}) - \tensor{B} ),
$$
where $\tensor{B} \in \R^{n_1 \times \cdots \times n_d}$ is the tensor that we want to approximate. In this notation, \refeqn{TAP3} is written as $ \min_{ \tuple{p}\in\Var{S}^{\times r}} f(\tuple{p})$. We propose a RGN method with \emph{trust region} \cite[section 7]{AMS2008} to solve this non-linear least squares problem. The key innovation is the addition of a scheme we call \emph{hot restarts} for effectively dealing with ill-conditioned Hessian approximations; see \refsec{sec_hot_restarts}.

Recall that in RGN methods with trust region the trust region subproblem (TRS) consists of (approximately) solving
\begin{align}\label{eqn_tr_model}
\min_{ \vect{p} \in B_{\Delta} } \, m_{\tuple{p}}(\vect{p}) \;\text{ with }\;  m_{\tuple{p}}(\vect{p}) := f(\tuple{p}) + \langle \deriv{\tuple{p}}{f}, \vect{p} \rangle + \frac{1}{2} \langle \vect{p}, \bigl( (\deriv{\tuple{p}}{\Phi})^* \circ \deriv{\tuple{p}}{\Phi}\bigr)( \vect{p} )\rangle,
\end{align}
where the trust region $B_{\Delta} := \{ \vect{p} \in \Tang{\tuple{p}}{\Var{S}^{\times r}} \;|\; \| \vect{p} \| \le \Delta \}$ with \emph{trust region radius} $\Delta > 0$. The high-level outline of RGN-HR is stated as Algorithm \ref{alg_rlm}. Aside from the addition of step 2, this is a conventional RGN with trust region method; see \cite[chapter 7]{AMS2008}. The reader familiar with RTR methods can skim the next subsections, where we specialize the necessary components to the TAP. The interesting constituents, namely the hot restarts technique, hidden in step 7, and the novel retraction operator in step 8 are described in \cref{sec_hot_restarts,sec_retraction} respectively.

\begin{algorithm2e}[t]\footnotesize
\SetAlgoLined
\KwData{Retraction $R$ on $\Var S^{\times r}$.}
\KwIn{The tensor to approximate $\tensor{B} \in \R^{n_1 \times \cdots \times n_d}$.}
\KwOut{A sequence of iterates $\tuple p^{(k)}\in\Var S^{\times r}$.}
Choose random initial points $p_i \in \Var{S} \subset \R^\Pi$\;
Solve the least-squares problem $[ p_i ]_{i=1}^r \vect{x} = \operatorname{vec}(\tensor{B})$ for $\vect{x}$, and set $p_i \leftarrow x_i p_i$\;
Let $\tuple{p}^{(1)} \leftarrow (p_1,\ldots,p_r)$, and set $k \leftarrow 0$\;
Choose a trust region radius $\Delta > 0$\;
\While{the method has not converged}{
Compute the gradient and GN Hessian approximation\;
Solve the TRS \refeqn{eqn_tr_model} for the search direction $\vect{p}_k \in B_{\Delta} \subset \Tang{\tuple{p}}{\Var{S}^{\times r}}$\;
Compute the tentative next iterate $\tuple{p}^{(k+1)} \leftarrow R_{\tuple{p}^{(k)}}(\vect{p}_k)$\;
Compute the trustworthiness $\rho_k$\;
Accept or reject the next iterate based on the trustworthiness $\rho_k$. If the step was accepted, increment $k$\;
Update the trust region radius $\Delta$.
}
\caption{RGN method with trust region for the TAP} \label{alg_rlm}
\end{algorithm2e}

\subsection{Choice of parameterization}\label{sec_parametrization}
We can choose \emph{any} convenient representation of points on $\Var{S}^{\times r}$.
An efficient data-sparse choice consists of representing a rank-$1$ tensor $p_i = \alpha_i \sten{a}{i}{1} \otimes \cdots \otimes \sten{a}{i}{d}$, where $\sten{a}{i}{k} \in \mathbb{S}^{n_k-1}$, as $(\alpha_i, \sten{a}{i}{1}, \ldots, \sten{a}{i}{d})$. However, representing an element $\vect{v} \in \mathbb{S}^{n} \subset \R^{n+1}$ using $n$ parameters is not convenient from a programming perspective.
For this reason, we prefer \emph{norm-balanced} representatives where $p_i = \sten{a}{i}{1} \otimes \cdots \otimes \sten{a}{i}{d}$ is represented as the tuple $\tuple{p}_i = (\sten{a}{i}{1},\ldots,\sten{a}{i}{d})$ with $\sten{a}{i}{k} \in \R^{n_k}$ and $\|\sten{a}{i}{1}\| = \cdots = \|\sten{a}{i}{d}\|$. This choice incurs a minor cost of $d-1$ parameters relative to minimal parameterizations. The CPD $\tuple{p} \in \Var{S}^{\times r}$ is then represented sparsely using $r(\Sigma+d)$ parameters as $\tuple{p} = (\tuple{p}_1, \ldots, \tuple{p}_r)$.

We can also choose a convenient basis for the tangent space of $\Var{S}^{\times r}$. Suppose that $p_i = \alpha_i \sten{a}{i}{1} \otimes \cdots \otimes \sten{a}{i}{d}$ are the rank-$1$ tensors, where $\sten{a}{i}{k} \in \mathbb{S}^{n_k-1}$. There is a specific orthonormal basis of $\Tang{p_i}{\Var{S}}$ that enables computationally efficient implementations of the basic operations, which are discussed in \refapp{app_implementation}. For $k=1$, let $U_{i,1} := I_{n_1}$ be the standard basis. For $k=2,\dots,d$, we choose an orthonormal basis $U_{i,k}$ of $\Tang{\sten{a}{i}{k}}{\mathbb{S}^{n_k-1}}$ as the first $n_k-1$ columns of the $Q$ factor of a rank-revealing QR-decomposition \((I_{n_k} - \sten{a}{i}{k} (\sten{a}{i}{k})^T)P = QR\) with $Q^T Q = I_{n_k}$, $R$ upper triangular, and $P$ a permutation matrix. Then the columns of
\begin{equation} \label{eqn_def_tangent_segre_repr}
 T_{p_i}
 :=
 \begin{bmatrix} U_{i,1}\otimes\sten{a}{i}{2}\otimes\cdots\otimes\sten{a}{i}{d} & \cdots & \sten{a}{i}{1}\otimes\cdots\otimes\sten{a}{i}{d-1}\otimes U_{i,d} \end{bmatrix}
 = \begin{bmatrix} T_{i,1} & \cdots & T_{i,d} \end{bmatrix}
\end{equation}
form an orthonormal basis of $\Tang{p_i}{\Var{S}}$. This $\Pi \times (\Sigma+1)$ matrix is \emph{never computed explicitly}, but rather we exploit its tensor structure. It suffices to store the matrices $U_{i,k} \in \R^{n_k \times (n_k-1)}$ for $k=2,\ldots,d$ and $i=1,\ldots,r$.

\subsection{Random starting points}
We choose the initial points $p_i \in \Var{S}$ in step {1} randomly by sampling the elements of $\sten{a}{i}{k}$, $i=1,\ldots,r$ and $k=1,\ldots,d$, i.i.d. from a standard normal distribution. Then, $p_i := \sten{a}{i}{1} \otimes \cdots \otimes \sten{a}{i}{d} \in \R^{n_1 \times \cdots \times n_d}$, which is represented sparsely with norm-balanced representatives. After the above random initialization, we solve the linear least-squares problem
\begin{align}\label{eqn_optimal_coefficients}
 \min_{x_1,\ldots,x_r\in \R} \Bigl\| \sum_{i=1}^r x_i p_i - \tensor{B} \Bigr\|_F
\end{align}
to determine the optimal linear combination of the initial $p_i$'s. This least-squares problem admits a unique solution with probability $1$, because the $p_i$'s are generically linearly independent; see, e.g., \cite[Corollary 4.5]{COV2017}. All of the standard methods for solving \refeqn{eqn_optimal_coefficients} can be employed; in \refapp{app_implementation} we describe how this operation is implemented efficiently without computing the $p_i$'s explicitly. Letting $x_1^*, \ldots, x_r^*$ denote the optimal coefficients, we take $x_i^* p_i \in \Var{S}$ rather than $p_i$ as ``optimally scaled'' rank-$1$ tensors in the initial rank-$r$ approximation of $\tensor{B}$.

The main motivation for this extra step is that it can be implemented more efficiently than one solve of the TRS. We observed that this modification generally reduces the number of iterations by a small amount. For the same reason, we also employed this modification in the hot restarts strategy discussed in \refsec{sec_hot_restarts}.

\subsection{The trust region scheme}
The trust region radius $\Delta$ is updated according to the standard scheme \cite[chapter 4]{NW2006}.
Define minimum and maximum radii as
\[
 \Delta_{\min}(\tuple{p}) := 10^{-1} \cdot \sqrt{\frac{d}{r} \cdot \sum_{i=1}^r \|\sten{a}{i}{1}\|_F^2}
 \;\;\text{ and }\;\;
 \Delta_{\max} := \frac{1}{2} \| \tensor{B} \|_F,
\]
respectively. As initial radius in step {4} we take
\(
 \Delta = \min\{ \Delta_{\min}(\tuple{p}^{(1)}), \Delta_{\max} \}.
\)

In step {9}, the trustworthiness of the model at $\tuple{p}^{(k)} \in \Var{S}^{\times r}$ is evaluated as
\[
 \rho_{k} = \frac{f(\tuple{p}^{(k)}) - f(\tuple{p}^{(k+1)})}{m_{\tuple{p}^{(k)}}(0) - m_{\tuple{p}^{(k)}}(\vect{p}_k)}.
\]
where $\vect{p}_k$ is the solution of the TRS in step {7}, and $\tuple{p}^{(k+1)}$ is computed in step {8}.
If the trustworthiness exceeds $0.2$, then $\tuple{p}^{(k+1)}$ is accepted as the next iterate in step {10}; otherwise, the iterate is rejected.

In step 11, the trust region radius is updated as follows. If the trustworthiness $\rho_k > 0.6$, then the trust region is enlarged by taking
\(
 \Delta = \min\{ 2 \|\vect{p}_k\|, \Delta_{\max} \},
\)
where $\vect{p}_k$ is the solution of the TRS in step {7}. Otherwise, the radius is adjusted as
\[
 \Delta = \min \biggl\{ \Bigl( \frac{1}{3} + \frac{2}{3} \cdot \bigl(1 + e^{-14\cdot(\rho_k - \frac{1}{3})} \bigr)^{-1} \Bigr) \Delta,\; \Delta_{\max} \biggr\},
\]
which is the strategy in Tensorlab's \texttt{cpd\_nls} method \cite{Tensorlab} with different constants. The effect of the scaling factor $\tfrac{1}{3} + \tfrac{2}{3} (1 + e^{-14(\rho_k-\frac{1}{3})})^{-1}$ is shown below:
\begin{center}
\begin{tikzpicture}
	\begin{axis}[
		xlabel=$\rho_k$,
		ylabel={},
		xmin=-.6, xmax=1, ymin=0.33, ymax=1,
		width=.8\textwidth, height=10em,
		grid=major
	]
	\addplot[color=black,samples=300] {1/3 + (2/3) / (1 + 2.718^(-14*(x-1/3)))};
	\end{axis}
\end{tikzpicture}
\end{center}
When $\rho_k < 0.2$, the trust region radius is strongly reduced by multiplying with approximately $\tfrac{1}{3}$, between $0.2 \le \rho_k \le 0.47$ the reduction factor increases approximately linearly with $\rho_k$, and finally for $0.47 \le \rho_k \le 0.6$ the radius is kept nearly constant.

\subsection{Gradient and Hessian approximation in coordinates} \label{sec_grad_hess_local_coords}
For solving the TRS in step 7, we need explicit expressions for the gradient and approximate Riemannian Hessian of the objective function. The derivative of $f$ at $\tuple{p} \in \Var{S}^{\times r}$ is
\[
 \deriv{\tuple{p}}{f} = \frac{1}{2} \deriv{\tuple{p}}{\langle F(\tuple{p}), F(\tuple{p})\rangle} = \langle \deriv{\tuple{p}}{F}, F(\tuple{p}) \rangle = \langle \deriv{\tuple{p}}{\Phi}, F(\tuple{p}) \rangle.
\]
Let $\tuple{p} = (p_1,\ldots,p_r)$ and $T_{p_i}$ be defined as in \refeqn{eqn_def_tangent_segre_repr}. With respect to the standard basis on $\R^\Pi$ and the orthonormal basis given by the columns of
\begin{align}\label{eqn_orth_basis}
 B := \operatorname{diag}(T_{p_1}, \ldots, T_{p_r})
\end{align}
on the domain, the derivative $\deriv{\tuple{p}}{f}$ is represented by the \emph{gradient} vector
\[
  \nabla_{\tuple{p}} f := T_{\tuple{p}}^T  F(\tuple{p}) = T_{\tuple{p}}^T \cdot \operatorname{vec}\bigl( \Phi(\tuple{p}) - \tensor{B} \bigr),
\]
where $\tensor{B} \in \R^{n_1 \times \cdots \times n_d}$ and {$T_{\tuple{p}}$ is the ``Jacobian'' matrix of $\deriv{\tuple{p}}{\Phi}$, i.e.,}
\begin{align} \label{eqn_jacobian}
 T_{\tuple{p}} := \begin{bmatrix} T_{p_1} & \cdots & T_{p_r} \end{bmatrix}
\end{align}
with $T_{p_i}$ as in \refeqn{eqn_def_tangent_segre_repr}. \refapp{sec_fast_gradient} gives an efficient algorithm for computing $\nabla_{\tuple{p}} f$.

We approximate the Riemannian Hessian by the GN approximation, namely $(\deriv{\tuple{p}}{F})^* (\deriv{\tuple{p}}{F})$; see, e.g., \cite[section 8.4]{AMS2008}. It follows from the above computations that its matrix with respect to the basis $B$ on both the domain and image is
\(
T_{\tuple{p}}^T T_{\tuple{p}}.
\)
Efficiently constructing it is covered in \refapp{sec_fast_hessian}.

Let $\vect{t} = B \vect{x} \in \Tang{\tuple{p}}{\Var{S}^{\times r}}$, {and define
\[
 \vect{r}_{\tuple{p}} := \operatorname{vec}\bigl( \Phi(\tuple{p}) - \tensor{B} \bigr), \quad
 \vect{g}_{\tuple{p}} := T_{\tuple{p}}^T \vect{r}_{\tuple{p}}, \text{ and} \quad
 H_{\tuple{p}} := T_{\tuple{p}}^T T_{\tuple{p}}.
\]}
Then, we find in coordinates with respect to the orthonormal basis $B$ that
\begin{align*}
\langle \deriv{\tuple{p}}{f}, \vect{t} \rangle &= \langle B T_{\tuple{p}}^T \vect{r}_{\tuple{p}}, B \vect{x} \rangle = \vect{g}_{\tuple{p}}^T \vect{x}, \text{ and } \\
\langle \vect{t}, \bigl( (\deriv{\tuple{p}}{\Phi})^* \circ \deriv{\tuple{p}}{\Phi}\bigr)( \vect{t} )\rangle &=
\langle B \vect{x}, B T_{\tuple{p}}^T T_{\tuple{p}} B^T B \vect{x} \rangle = \vect{x}^T H_{\tuple{p}} \vect{x}.
\end{align*}

\subsection{Solving the TRS} \label{sec_tr_subproblem}
The key step in RTR methods is solving the TRS in \refeqn{eqn_tr_model} for obtaining a suitable tangent direction along which the retraction can proceed. Many standard strategies exist for computing or approximating the solution of the TRS \cite{NW2006}. We choose the cheap dogleg heuristic \cite[section 4.1]{NW2006}.

Note that $H_{\tuple{p}}$ is symmetric positive semi-definite. Then the \emph{Newton direction} $\vect{p}_{\text{N}}$ and the \emph{Cauchy point} are given respectively by
\begin{align} \label{eqn_GN_and_CP}
 H_{\tuple{p}} \vect{p}_{\text{N}} := -\vect{g}_{\tuple{p}}
 \;\text{ and }\;
\vect{p}_{\text{C}} = - \frac{\vect{g}_{\tuple{p}}^T H_{\tuple{p}} \vect{g}_{\tuple{p}}}{\vect{g}_{\tuple{p}}^T \vect{g}_{\tuple{p}}} \vect{g}_{\tuple{p}}.
 \end{align}
If $\vect{p}_{\text{N}}$ is outside of the trust region $B_{\Delta}$, but $\vect{p}_{\text{C}} \in B_{\Delta}$, then the dogleg step is the intersection with the boundary of the ball $B_{\Delta}$ of vector pointing from $\vect{p}_{\text{C}}$ to $\vect{p}_{\text{N}}$:
\[
 \vect{p}_{\text{I}} := \vect{p}_{\text{C}} + (\tau-1)(\vect{p}_{\text{N}} - \vect{p}_{\text{C}}),
\]
where $1 \le \tau \le 2$ is the unique solution such that $\| \vect{p}_{\text{I}} \|^2 = \Delta^2$.
Then, the dogleg step approximates the optimal TRS solution of \refeqn{eqn_tr_model} by the next rule:
\begin{align} \label{eqn_dogleg_step}
\widehat{\vect{p}} =
 \begin{cases}
  \vect{p}_{\text{N}} & \text{if } \| \vect{p}_{\text{N}} \| \le \Delta, \\
  \vect{p}_{\text{C}} & \text{if } \| \vect{p}_{\text{N}} \| > \Delta \text{ and } \| \vect{p}_{\text{C}} \| \ge \Delta, \\
  \vect{p}_{\text{I}} & \text{otherwise}.
 \end{cases}
\end{align}

As mentioned in \refrem{rem_numerical_accuracy}, computing $\vect{p}_{\text{N}}$ via the normal equations is not numerically stable; however, a stable least-squares solve via a $QR$-decomposition is computationally not attractive. We choose solving the normal equations as this requires much fewer operations due to $H_{\tuple{p}}$'s structure; see \refapp{app_implementation}.

\subsection{Stopping criterion} \label{sec_stopping_criterion}
The RGN-HR method is halted based on a multi-com\-po\-nent stopping criterion involving $5$ parameters; the first $4$ are standard: an absolute tolerance $\tau_f$ on the objective function value, a relative tolerance $\tau_{\Delta f}$ on the improvement of the objective function value, a relative tolerance on the step size $\tau_{\Delta x}$, and a maximum number of iterations $k_{\max}$. The last parameter $r_{\max}$ specifies the maximum number of restarts (see \refsec{sec_hot_restarts}).

Specifically, the method halts in the $k$th iteration if either
\[
 f(\tuple{p}^{(k)}) \le \tau_f,
 \quad
 \frac{|f(\tuple{p}^{(k-1)})-f(\tuple{p}^{(k)})|}{f(\tuple{p}^{(1)})} \le \tau_{\Delta f},
 \text{ or }\quad
 \frac{\|\vect{p}_k\|}{ \bigl( \sum_{k=1}^d \|A_k\|^2_F \bigr)^{\frac{1}{2}}} \le \tau_{\Delta x},
\]
where $A_k = [\sten{a}{i}{k}]_{i=1}^r$ are the factor matrices corresponding to the norm-balanced representatives.
If neither of these conditions is satisfied after $k_{\max}$ iterations or after $r_{\max}$ restarts, then the method halts.

\section{The hot restarts strategy}\label{sec_hot_restarts}
The main motivation for adding the trust region scheme to (quasi-)\-Newton methods consists of obtaining a globally convergent method that still has a local superlinear rate of convergence.
In principle, the trust region mechanism can handle ill-conditioned Hessian approximations $H_{\tuple{p}} = T_{\tuple{p}}^T T_{\tuple{p}}$, where $T_{\tuple{p}}$ is as in \refeqn{eqn_jacobian}, without special considerations. However, we observed in practice that the RGN with trust region method progressed very slowly near ill-conditionened $H_{\tuple{p}}$'s.

In \cref{sec_ill_H_p}, we investigate the geometry of decompositions $\tuple{p}$ inducing singular Hessian approximations $H_{\tuple{p}}$. The main conclusion is that \emph{these decompositions should be avoided altogether.} We believe that one should generally escape ill-conditioned Hessian approximations as quickly as possible. Several techniques for this were considered before in the context of the CPD, notably by extrapolating previous search directions \cite{AB2000,RCH2008,DeSterck2012,CHQ2011}, using stochastic gradient-descent algorithms, using the true Hessian, or exploiting third-order information \cite{NP2006,AG2016}. Notwithstanding extensive experimentation,\footnote{We performed preliminary experiments with regularization, using the true Hessian, switching to steepest descent or conjugate gradients near ill-conditioned decompositions, and a simple randomization procedure wherein a multiple of a random vector is added to the dogleg direction in which the weight increases as the decomposition becomes more ill-conditioned.} we found that the following Monte Carlo approach typically resulted in the best results: randomly sample nearby points of $\tuple{p}$ on $\Var{S}^{\times r}$ until one is found with a reasonably well-conditioned Hessian approximation. The specific details of {this \emph{hot restarts} scheme are stated} in \refalg{alg_hot_restarts}, which is executed before computing the dogleg step in \refeqn{eqn_dogleg_step}.

\begin{algorithm2e}[tb]\footnotesize
\caption{Hot restarts} \label{alg_hot_restarts}
 Let $(\sten{a}{i}{1}, \ldots, \sten{a}{i}{d})$ be a norm-balanced representative of $p_i$\;
 $t \leftarrow 1$\;
 $\widehat{\alpha} \leftarrow \min \bigl\{ \frac{1}{4}, 10 \cdot \frac{\|\tensor{B}-\Phi(\tuple{p})\|_F}{\|\tensor{B}\|_F} \bigr\}$\;
 Compute the Cholesky decomposition $H_{\tuple{p}} = L_{\tuple{p}} L_{\tuple{p}}^T$\;
 \While{Cholesky decomposition failed \textbf{\emph{or}} $\min_{i} (L_{\tuple{p}})_{ii} < 10^{-5}$}{
    $\alpha \leftarrow t \cdot \widehat{\alpha}$\;
    \For{$i = 1, \ldots, r$}{
       Let the elements of $\vect{n}_k \in \R^{n_k}$ be sampled i.i.d. from $N(0,1)$\;
       $p_i' \leftarrow \bigl((1-\alpha)\sten{a}{i}{1} + \alpha \frac{\|\sten{a}{i}{1}\|}{\|\vect{n}_1\|} \vect{n}_1 \bigr) \otimes \cdots \otimes \bigl((1-\alpha)\sten{a}{i}{d} + \alpha \frac{\|\sten{a}{i}{d}\|}{\|\vect{n}_d\|} \vect{n}_d \bigr)$\;
    }
    Solve $\min_{\vect{x}} \| \tensor{B} - \sum_{i=1}^r x_i p_i'\|_F$\;
    $\tuple{p} \leftarrow (x_1 p_1', x_2 p_2', \ldots, x_r p_r')$\;
    $\Delta \leftarrow \min\{ \Delta_{\min}(\tuple{p}), \Delta_{\max} \}$\;
    $t \leftarrow t + 1$\;
    Compute the Cholesky decomposition $H_{\tuple{p}} = L_{\tuple{p}} L_{\tuple{p}}^T$\;
 }
\end{algorithm2e}

\begin{remark}[Global convergence]
Proving global convergence for any input $\tensor{B} \in \R^{n_1 \times \cdots \times n_d}$ is impossible.
Assume that a hot restart is triggered whenever the CPD $\tuple{p}^{(k)} \in \Var{S}^{\times r}$ in the execution of \refalg{alg_rlm} has $\kappa(\tuple{p}^{(k)}) \ge \tau > 1$. Let $\Var{G}_{\le \tau} := \{ \tuple{p}\in\Var{S}^{\times r} \;|\; \kappa(\tuple{p}) \le \tau \}$ be the locus of decompositions whose condition number is below the jumping threshold. Since \refalg{alg_rlm} produces a sequence of points $\tuple{p}^{(k)} \in \Var{G}_{\le \tau}$, it can at most converge to local minimizers of \refeqn{TAP3} that are in $\Var{G}_{\le \tau}$. Hence, we can at most prove global convergence for inputs $\tensor{B} \in \R^{n_1 \times \cdots \times n_d}$ whose best rank-$r$ approximation $\tensor{B}^*$ exists and is such that $\Phi^{-1}(\tensor{B}^*)$ has at least one CPD contained in $\Var{G}_{\le \tau}$. A major technical obstacle then still remains because the hot restarts procedure will introduce a stochastic aspect, further complicating any proof of global convergence. A full investigation is outside of the scope of this paper.
\end{remark}

\begin{remark}[Local convergence]
Local convergence is proved as follows. Let $\tau$ be as above. Taking $\frac{\kappa}{\tau} < \alpha < 1$ in Theorem 1 of \cite{BV2017bis} establishes conditions for linear or quadratic local convergence, including attraction radii and multiplicative constants. If $\frac{\kappa}{\tau} \ge 1$, then there is no proof of convergence, as expected.
\end{remark}

\subsection{Recognizing ill-conditioned decompositions}
For determining if a decomposition $\tuple{p}$ is ill-conditioned, one could compute or approximate the smallest singular value of $H_{\tuple{p}}$.
For computational efficiency, however, we prefer relying on properties of the Cholesky decomposition $H_{\tuple{p}} = L_{\tuple{p}} L_{\tuple{p}}^T$, where $L_{\tuple{p}}$ is lower triangular. Recall that $H_{\tuple{p}}=T_{\tuple{p}}^T T_{\tuple{p}}$ and that $T_{\tuple{p}}$ is the matrix of $\deriv{\tuple{p}}{\Phi}$.
By Lemma \ref{norm_of_phi}, $\Vert \deriv{\tuple{p}}{\Phi}\Vert_2\geq 1$ and so $\varsigma_{\max}(H_{\tuple{p}}) \ge 1$. Combining this with the fact that $\varsigma_k(H_{\tuple{p}}) = \bigl( \varsigma_k(L_{\tuple{p}}) \bigr)^2$ and that one has $\varsigma_{\min}(A) \le |\lambda_{\min}(A)|$, where $\lambda_{\min}$ is the smallest eigenvalue in magnitude, we find
$$
\kappa_2( H_{\tuple{p}} ) = \frac{\varsigma_{\max}( H_{\tuple{p}} )}{\varsigma_{\min}( H_{\tuple{p}} )} \ge \bigl( \varsigma_{\min}(L_{\tuple{p}}) \bigr)^{-2}
\ge \max_i |(L_{\tuple{p}})_{i,i}|^{-2}.
$$
In the last step we used that the diagonal entries of a triangular matrix are its eigenvalues.
Moreover, a breakdown will occur while computing the Cholesky decomposition if $H_{\tuple{p}}$ is very ill-conditioned \cite[p. 200]{Higham1996}.
For these reasons, we apply a hot restart if a breakdown occurred or if one of the diagonal elements of $L_{\tuple{p}}$ is less than $10^{-5}$, so $\kappa_2(H_{\tuple{p}}) \ge 10^{10}$; see line {5} of \refalg{alg_hot_restarts}.
After executing \refalg{alg_hot_restarts}, the Cholesky decomposition is reused for efficiently solving $L_{\tuple{p}} L_{\tuple{p}}^{T} \vect{p}_{\mathrm{N}} = - \vect{g}_{\tuple{p}}$ in the dogleg step via backsubstitutions.

The constants appearing in line $3$ of \refalg{alg_hot_restarts} were empirically chosen.
We choose $\widehat{\alpha}$ in function of the relative residual, for the following reason. If the relative residual $\tau = \tfrac{\|\tensor{B}-\Phi(\tuple{p})\|_F}{\|\tensor{B}\|_F}$ is small, one hopes to be close to a global optimizer, so spending computational resources to search an acceptable decomposition in a small radius about $\tuple{p}$ is justified. However, if ill-conditioning is already encountered when $\tau \gtrsim 10^{-1}$, then our experiments suggested that aggressively restarting, {i.e., Monte Carlo sampling with a large $\widehat{\alpha}$,} was the most cost-effective strategy.
As long as a satisfyingly well-conditioned decomposition is not sampled, $\alpha$ is gradually increased.

\subsection{Analysis}\label{sec_ill_H_p}
In this subsection, we argue why CPDs with a singular Hessian approximation $H_{\tuple{p}}$ are troublesome for the RGN method. The arguments below are not intended to be a proof, but rather they provide an intuition and offer an outline of a potential proof strategy. A complete proof is beyond the scope of this work.

From the properties of Gram matrices, we have the following connection between $\kappa(\tuple{p})$ from \refeqn{kappa} and $H_{\tuple{p}}$:
\begin{align} \label{eqn_condition_number_hessian}
\kappa(\tuple{p}) =
\begin{cases}
 \infty & \text{if } H_{\tuple{p}} \text{ is singular}, \\
 \| T_{\tuple{p}}^\dagger \|_2 = \| H_{\tuple{p}}^{-1} \|_2^{\frac{1}{2}} & \text{otherwise}.
\end{cases}
\end{align}
Assume that some computed local optimizer $\tuple{p}^*$ of \refeqn{TAP3} is well-conditioned, and hence, by \refeqn{eqn_condition_number_hessian}, the associated $H_{\tuple{p}^*}$ is reasonably well-conditioned. However,
during the execution of the RTR method, we may encounter points $\tuple{p}$, for which $H_{\tuple{p}}$ is almost singular. We sketch an informal argument why escaping the neighborhood of such points may take many iterations in a RGN method, with or without trust region.

Following \cite[equation (2.3)]{BV2017}, \emph{the locus of ill-posed CPDs} is defined as follows:
\[
\Var{I}_r := \{ \tuple{q} \in \Var{S}^{\times r} \;|\; \kappa(\tuple{q}) = \infty \} = \{ \tuple{q} \in \Var{S}^{\times r} \;|\; H_{\tuple{q}}=T_{\tuple{q}}^TT_{\tuple{q}} \text{ is singular} \};
\]
the second equality is by \refeqn{eqn_condition_number_hessian}. The following lemma, proved in \refapp{app_proof_sthosvd_retraction}, implies that a randomly chosen CPD $\tuple{p}$ lies on $\Var I_r$ with probability zero, but also that
there is a nonzero probability of coming close to $\Var I_r$; we do not know how large it is.

\begin{lemma}\label{prop_I_r}
If $\Sec{r}{\Var{S}}$ is non-defective, then $\Var I_r$ is contained in a non-empty, strict  subvariety of $\Var S^{\times r}$.
\end{lemma}

Let $\tuple{p}, \tuple{q} \in \Var{S}^{\times r}$ with $\tuple{q}$ close to $\tuple{p}$, and assume that the Hessian approximation $H_{\tuple{q}}$ is singular, i.e., $\tuple{q} \in \Var{I}_r$.
Let $T_{\tuple{p}} = U S V^T$ be the compact SVD of $T_{\tuple{p}}$, where $V$ is orthogonal, $U$ has orthonormal columns, and $S = \operatorname{diag}(\varsigma_1, \varsigma_2, \ldots, \varsigma_{r(\Sigma+1)})$ with $\varsigma_1 \ge \varsigma_2 \ge \ldots \ge \varsigma_{r(\Sigma+1)} \ge 0$. Let $n := r(\Sigma+1)$ and write
\(
V=\begin{bmatrix}\vect{v}_1 & \cdots & \vect{v}_n \end{bmatrix}
\). Assuming $\varsigma_{n} > 0$, or equivalently $\kappa(\tuple{p}) < \infty$, by \refeqn{eqn_GN_and_CP}, the Newton direction is
\begin{equation}\label{alpha_i}
\vect{p}_{\mathrm{N}} = -H_{\tuple{p}}^\dagger \vect{g}_{\tuple{p}} = - V S^{-1} U^T \vect{r}_{\tuple{p}} = -\sum_{i=1}^{n} \alpha_i \varsigma_i^{-1} \vect{v}_i,
\end{equation}
where $\alpha_i := (U^T\vect{r}_{\tuple{p}})_i$. Since $\tuple{p}$ and $\tuple{q}$ are close,
$
\varsigma_n^2 = \|H_{\tuple{p}}^{-1}\|_2^{-1} \approx \|H_{\tuple{q}}^{-1}\|_2^{-1} = 0,
$
and so $\vect{p}_{\mathrm{N}}$ is pointing mostly in the direction of $\vect{v}_{n}$. As long as $\vect{v}_n$ is a descent direction of the objective $\tfrac{1}{2}\|F(\tuple{p})\|^2$, the dogleg step is expected to yield an escape from~$\tuple{q}$.

Unfortunately, there are two classes of tensors, whose decompositions could be particularly troublesome when coming close to them during the RGN method: tensors with infinitely many decompositions and tensors whose \emph{border ranks} \cite[section 2.4]{Landsberg2012} are not equal to their ranks.\footnote{It is an interesting open question whether these decompositions could actually be attractive for the RGN method without trust region.}

\subsubsection{Tensors with infinitely many decompositions}
Consider a decomposition $\tuple{q}$ that is not isolated. To show this is troublesome, we need the next result.

\begin{proposition}\label{prop_curve_is_illcond}
Let $\tuple{q}=(q_1,\ldots,q_r) \in \Var{S}^{\times r}$ and $\tensor{B} = \Phi(\tuple{q})$. Let $\Phi^{-1}(\tensor{B})$ contain a positive-dimensional analytic submanifold $\Var{E}$ without boundary with $\tuple{q}\in \Var{E}$.
\begin{enumerate}
\item We have $\Var{E} \subset \Var{I}_r$.
\item If
$\tuple{z} = (z_1, z_2, \ldots, z_r) \in \Tang{\tuple{q}}{\Var{E}} \subset \Tang{\tuple{q}}{\Var{S}^{\times r}}$
and $\vect{z}_i$ are the local coordinates of $z_i$ with respect to the basis of $\Tang{q_i}{\Var{S}}$ given by the columns of the matrix $T_{q_i}$ from~\refeqn{eqn_def_tangent_segre_repr}, then
\(
\vect{z}^T := \begin{bmatrix}
\vect{z}_1^T & \vect{z}_2^T & \cdots & \vect{z}_r^T
\end{bmatrix} \in \ker T_{\tuple{q}}.
\)
\end{enumerate}
\end{proposition}
\begin{proof}
By construction, $\Var{E}$ is a smooth algebraic variety. From \cite[Lemma 6.5]{COV2017} it follows that for every $\tuple{p}=(p_1,\ldots,p_r) \in \Var{E}$ the corresponding $\langle \Tang{p_1}{\Var{S}}, \ldots, \Tang{p_r}{\Var{S}} \rangle$ is of dimension strictly less than $r \cdot \dim \Var{S} \le \Pi$. This entails that $T_{\tuple{p}}$ has linearly dependent columns, proving the first part. Furthermore, $0=(\deriv{\tuple{q}}{\Phi})( \tuple{z} )=z_1 + z_2 + \dots + z_r$, which written in coordinates is
$
 T_{\tuple{q}} \vect{z} = T_{q_1} \vect{z}_1 + T_{q_2} \vect{z}_2 + \cdots + T_{q_r} \vect{z}_r = 0,
$
by \refeqn{eqn_def_tangent_segre_repr} and \refeqn{eqn_jacobian}.
\end{proof}

Consequently, the kernel of $H_{\tuple{q}}=T_{\tuple{q}}^TT_{\tuple{q}}$ contains at least the tangent space $\Tang{\tuple{q}}{\Var E}$. The implication for nearby points $\tuple{p} \approx \tuple{q}$ is as follows. Let $V=[\begin{smallmatrix} \vect{v}_1 & \cdots & \vect{v}_n\end{smallmatrix}]$ be the matrix of right singular vectors of $T_{\tuple p}$, and $e := \dim \Var{E}$. Then, because $\tuple{p}\approx \tuple{q}$ and since the kernel of $T_{\tuple q}$ is at least $e$-dimensional, we expect the linear space $\langle \vect{v}_{n-e+1},\ldots, \vect{v}_n \rangle$ to be almost parallel to $\Tang{\tuple{q}}{\Var E}$. Moreover, at least $e$ singular values of $T_{\tuple p}$ will be small, so that, by \refeqn{alpha_i}, the Newton-direction $\vect{p}_\mathrm{N}$ is significantly parallel to $\Tang{\tuple p}{\Var E}.$
Retracting along $\vect{p}_{\text{N}}$ or the dogleg step is then expected to move $\tuple{p}$ to $\tuple{p}'$ along a path that could be mostly parallel to $\Var{E}$. The resulting point $\tuple{p}'$ would then also lie close to $\Var{E}$, i.e., close to one of the other decompositions of $\Phi(\tuple{q})$. Iterating this reasoning suggests that points close to $\Var{E}$ could be hard to escape with an RGN method.

\subsubsection{Open boundary tensors}
Let $\tensor{B}$ be an \emph{$r$-open boundary tensor}: $\tensor{B}$ has border rank \cite[section 2.4]{Landsberg2012} equal to $r$, but $\mathrm{rank}(\tensor{B})>r$. This means that there exist sequences of tensors $\Phi(\tuple{p}_i) \to \tensor{B}$ as $i \to \infty$, where $\tuple{p}_i \in \Var{S}^{\times r}$. The rank-$r$ TAP is ill-posed in this case, as only an infimum exists \cite{dSL2008}. CPDs $\tuple{p} \in\Var{S}^{\times r}$ with $\Phi(\tuple{p})$ near $\tensor{B}$ are also troublesome for RGN methods, as we argue next.

Let $\tuple{p}(t)=(p_1(t), \ldots, p_r(t)) \subset \Var S^{\times r}$ be a smooth curve with $\lim_{t\to0} \Phi(\tuple{p}(t)) = \tensor{B}$. The proof of \cite[theorem 1.4]{BV2017} shows that scalar functions $\nu_1(t),\ldots,\nu_r(t)$ exist so that
\[
 \vect{h}(t) =  \bigl( \nu_1(t) p_1(t),\, \nu_2(t) p_2(t),\, \ldots,\, \nu_r(t) p_r(t) \bigr) \in \Tang{\tuple{p}(t)}{\Var S^{\times r}}
\]
tends to a nonzero vector $\vect{h}_\star\in\ker T$, where $T := \lim_{t\to0} T_{\tuple{p}(t)}$ with $T_{\tuple{p}(t)}$ as in \refeqn{eqn_jacobian}. This limit $T$ exists, as
\(
 T_{\tuple{p}(t)} \in \operatorname{St}(\Sigma+1, \R^\Pi)^{\times r}
\)
lives in the $r$-fold product of the Stiefel manifold of matrices in $\R^{\Pi \times (\Sigma+1)}$ with orthonormal columns, which is closed.

Assume now that $\tuple{p} := \tuple{p}(\epsilon)$ and $\vect{h}:=\vect{h}(\epsilon)$ for some specific $\epsilon \approx 0$. Whenever $\Vert\vect{h}- \vect{h}_\star\Vert$ is small, the Newton direction $\vect{p}_{\mathrm{N}}$ in \refeqn{alpha_i} is expected to have a significant component in the direction of $\vect{h}$. Note that the entries of $\vect{h}$ are constant multiples of the $p_i(\epsilon)$'s. Hence, adding any multiple of $\vect{h}$ to $\tuple{p}$ yields an element of~$\Var{S}^{\times r}$. Many retraction operators $R_x : \Tang{x}{\Var{M}} \to \Var{M}$ on a manifold $\Var{M}$, including both retractions discussed in \refsec{sec_retraction}, are the identity when $x + \vect{h} \in \Var{M}$. In this case,
\[
\tuple{p}' :=
R_{\tuple{p}}\bigl(\vect{h} \bigr) =\tuple{p} +  \vect{h} = ((1+\nu_1(\epsilon))p_1(\epsilon), \ldots, (1+\nu_r(\epsilon))p_r(\epsilon)).
\]
As a result, only the norms of the rank-$1$ terms in the CPD are altered by such a retraction along $\vect{h}$.
Since the matrix $T_{\tuple{p}}$ is invariant under this scaling \cite[Proposition~4.4]{BV2017}, we have
\(
 T_{\tuple{p}} = T_{\tuple{p}'}.
\)
Hence, the same argument also applies at the next iterate~$\tuple{p}'$, suggesting again that CPDs $\tuple{p}$ with $\Phi(\tuple{p})$ close to an $r$-open boundary tensor $\tensor{B}$ could be hard to escape with RGN methods.

\section{The product ST-HOSVD retraction} \label{sec_retraction}
A characteristic ingredient of a Riemannian optimization method is an effective retraction operator. It appears in step {8} of \refalg{alg_rlm} for pushing forward the current iterate $\tuple{p}^{(k)} \in \Var{S}^{\times r}$ along a tangent vector $\vect{p}_k \in \Tang{\tuple{p}^{(k)}}{\Var{S}^{\times r}}$. The result is the (tentative) next iterate $\tuple{p}' \in \Var{S}^{\times r}$.

A \emph{retraction} is a map taking a tangent vector $\xi_p \in \Tang{p}{\Var{M}}$ to a manifold $\Var{M}$ at $p$ to the manifold itself. It can be regarded as an operator that approximates the action of the exponential map to first order \cite{AMS2008}. A formal definition is as follows \cite{ADMMS2002,AMS2008}.
\begin{definition}\label{def_retraction}
Let $\Var{M}$ be a manifold. A \emph{retraction} $R$ is a map $\Var{T}\Var{M} \to \Var{M}$ that satisfies all of the following properties for every $p \in \Var{M}$:
\begin{enumerate}
\item $R(p,0_p) = p$;
\item there exists {an open} neighborhood $\Var{N} \subset \Var{T}{\Var{M}}$ of $(p,0_p)$ such that the restriction $R|_\Var{N}$ is well-defined and a smooth map;
\item local rigidity: $\deriv{0_x}{R(x,\cdot)} = \operatorname{Id}_{\Tang{x}{\Var{M}}}$ for all $(x, 0_x) \in \Var{N}$.
\end{enumerate}
We let $R_p(\cdot) = R(p,\cdot)$ denote the retraction $R$ with foot at $p$.
\end{definition}

Choosing an efficient retraction is critical for attaining good computational performance.
Fortunately, we can exploit the following well-known result that on a product manifold, the product of individual retractions specifies a retraction.
\begin{lemma} \label{lem_product_retraction}
 Let $\Var{M}_1,\ldots,\Var{M}_r$ be manifolds. Let $R_i : \Var{T}\Var{M}_i \to \Var{M}_i$ be retractions. Then the following is a retraction on $\Var{M}_1 \times \cdots \times \Var{M}_r$:
 \begin{align*}
  R\bigl( (p_1, \xi_{p_1}), \ldots, (p_r, \xi_{p_r}) \bigr) = (R_1(p_1, \xi_{p_1}), \ldots, R_r(p_r, \xi_{p_r})),
 \end{align*}
 where $p_i \in \Var{M}_i$ and $\xi_{p_i} \in \Tang{p_i}{\Var{M}_i}.$
\end{lemma}

Because $\Var{S}^{\times r}$ is a product manifold, it thus suffices to find an efficient retraction for the Segre manifold $\Var{S}$.
Since $\Var{S}$ coincides with the manifold of tensors of multilinear rank $(1,\ldots,1)$, we can apply the truncated higher-order singular value decomposition (T-HOSVD) retraction from \cite[Proposition 2.3]{KSV2014}. This retraction is defined as the rank-$(1,\ldots,1)$ T-HOSVD approximation of $\tensor{A} + \vect{p}$ \cite{Lathauwer2000}, i.e.,
\(
\widehat{R}_{\tensor{A}}( \vect{p} ) := (\widehat{\vect{q}}_1 \widehat{\vect{q}}_1^T, \ldots, \widehat{\vect{q}}_d \widehat{\vect{q}}_d^T) \cdot (\tensor{A} + \vect{p}),
\)
where $\tensor{A} \in \Var{S} \subset \R^{n_1 \times \cdots \times n_d}$, $\vect{p} \in \Tang{\tensor{A}}{\Var{S}} \subset \R^{n_1 \times \cdots \times n_d}$, and $\widehat{\vect{q}}_k \in \R^{n_k}$ is the dominant left singular vector of the \textit{standard flattening} $(\tensor{A} + \vect{p})_{(k)}$ \cite{Lathauwer2000}.

Instead of applying the above T-HOSVD retraction, we propose a retraction based on the sequentially truncated HOSVD (ST-HOSVD) \cite{Hackbusch2012,VVM2012}, as its computational complexity is lower \cite{VVM2012}.
For rank-$1$ tensors $\tensor{A} \in \Var{S} \subset \R^{n_1 \times \cdots \times n_d}$, we define a map $\Var{T}\Var{S} \to \Var{S}$ as follows
\begin{align} \label{eqn_sthosvd_retraction}
 R_{\tensor{A}}( \vect{p} ) := (\vect{q}_1 \vect{q}_1^T, \ldots, \vect{q}_d \vect{q}_d^T) \cdot (\tensor{A} + \vect{p}),
\end{align}
where $\vect{p} \in \Tang{\tensor{A}}{\Var{S}} \subset \R^{n_1 \times \cdots \times n_d}$, and the $\vect{q}_k$'s are the factors of the rank-$(1,\ldots,1)$ ST-HOSVD of $\tensor{A} + \vect{p}$. While it is not indicated in the notation, one should remember that the vector $\vect{q}_i$ depends on $\vect{q}_1, \ldots, \vect{q}_{i-1}$ \cite{VVM2012}. We have the next result.

\begin{lemma} \label{lem_sthosvd_retraction}
The map in \refeqn{eqn_sthosvd_retraction} defines a retraction $\Var{T}\Var{S} \to \Var{S}$.
\end{lemma}
\begin{proof}
 The proof is a verbatim copy of the proof of \cite[Proposition 2.3]{KSV2014}, where the necessary quasi-best approximation property is \cite[Theorem 10.5]{Hackbusch2012}.
\end{proof}

Let $\tuple{p} = (p_1,\ldots,p_r) \in \Var{S}^{\times r}$ and let $\vect{t} \in \Tang{\tuple{p}}{\Var{S}^{\times r}}$. The tangent vector $\vect{t}$ is given in coordinates by $\vect{x} \in \R^{r(\Sigma+1)}$ with respect to the orthonormal basis $B$ in \refeqn{eqn_orth_basis}. Then, the product retraction on $\Var{S}^{\times r}$ is
\[
 R_{\tuple{p}}(\vect{t}) := \bigl( R_{p_1}(T_{p_1} \vect{x}_1), \ldots, R_{p_r}(T_{p_r} \vect{x}_r) \bigr),
\]
where $\vect{x}^T = \begin{bmatrix} \vect{x}_1^T & \cdots & \vect{x}_r^T \end{bmatrix}$. An efficient implementation is given in \refsec{sec_fast_retraction}.

\begin{remark}
The above product ST-HOSVD retraction was proved valid only in a neighborhood $\Var{N} \subset \Var{T}\Var{S}^{\times r}$ about $(\tuple{p}, 0)$ by \refprop{lem_product_retraction} and \reflem{lem_sthosvd_retraction}. Before applying the retraction $R_{\tuple{p}}$ to $\vect{t}$, in theory, one should verify that $(\tuple{p}, \vect{t}) \in \Var{N}$. In practice, this strict regime is not enforced; retractions generally are also applied outside of the region where the smoothness of $R|_{\Var{N}}$ can be proved. In our implementation, the retraction is applied to all inputs, regardless of whether $(\tensor{A}, \vect{t}) \in \Var{N}$.
\end{remark}

\section{Numerical experiments} \label{sec_numerical_experiments}
RGN-HR was implemented\footnote{Our implementation can be obtained from \url{https://arxiv.org/abs/1709.00033}.} in Matlab, employing some functionality of Tensorlab v3.0 \cite{Tensorlab}.
All numerical experiments were conducted in Matlab R2017a using 2 computational threads on a computer system comprising two Intel Xeon E5-2697 v3 CPU's with 14 cores each (all clocked at 2.6GHz) and 128GB of total main memory.

We compare RGN-HR with state-of-the-art nonlinear least-squares (NLS) solvers designed for the TAP from Tensorlab v3 \cite{Tensorlab}, specifically the trust region method with dogleg steps, called \verb|nls_gndl|. It employs factor matrices as parameterization and uses the GN approximation of the Hessian of \refeqn{eqn_optimization_problem_classic}, i.e., $J^T J$ with $J$ as in \refeqn{eqn_overparam_jacobian}, as explained in \refsec{sec_stateoftheart}. As a result, the normal equations are ill-posed, and, hence, the Newton step $\vect{p}_{\text{N}}$ is not well defined. Therefore, Tensorlab (approximately) computes the minimum-norm solution of the least-squares problem $J \vect{p}_{\text{N}} = -\vect{r}$, where $\vect{r}$ is the residual. Both {direct} (\verb|LargeScale = false|) and {iterative} (\verb|LargeScale = true|) solution of this least-squares problem are supported. The direct algorithm computes the Moore-Penrose pseudoinverse and applies it to $-\vect{r}$, while the iterative method is LSQR with block Jacobi preconditioner; see \cite{Sorber2013a} and the user manual of Tensorlab \cite{Tensorlab} for details. We refer to them as GNDL and GNDL-PCG, respectively.
Tensorlab offers two additional NLS solvers (\texttt{nls\_lm} and \texttt{nls\_cgs}) that we did not find competitive.

GNDL is Tensorlab's default method for computing dense small-scale CPDs, i.e., $r(\Sigma+d) \le 100$, whereas GNDL-PCG is the default method for computing CPDs in the case $r(\Sigma+d) > 100$. It should be stressed that for fixed $d$, GNDL and RGN-HR both \emph{have the same asymptotic computational complexity}. Hence, this is a comparison between methods appropriate for the same class of TAPs, i.e., dense, small-scale, low-residual problems. On the other hand, one may not expect that RGN-HR will be competitive with GNDL-PCG for medium and large-scale problems with $r(\Sigma+d)$ much larger than $1000$. The reason is that GNDL-PCG was specifically designed for such problems by crudely approximating the Newton direction using an asymptotically faster algorithm. Indeed, the complexity per iteration of RGN-HR is roughly $r^3 (\Sigma+d)^3$, while GNDL-PCG's complexity is only $k_{\text{CG}} \cdot r^2 (\Sigma+d)^2$. The exact Newton direction is recovered for $k_{\text{CG}} = \mathcal{O}( r(\Sigma+d) )$. Recall from the introduction that small-scale dense TAPs often result from an orthogonal Tucker compression of a medium or large-scale tensor; hence, solving small-scale TAPs is an important computational kernel.

For {the above} reason, we experiment with dense, small-scale CPDs. We compare RGN-HR with GNDL and GNDL-PCG for some order-$3$ tensors; namely, in \refsec{sec_model1} for tensors in $\R^{15\times 15 \times 15}$ and in  \refsec{sec_model2} for tensors in $\R^{13\times 11 \times 9}$.

\subsection{Experimental setup}
Let $\tensor{A} \sim \Var{N}$ denote that the entries of $\tensor{A} \in \R^{n_1 \times \cdots \times n_d}$ are identically and independently distributed (i.i.d.)\ following a standard normal distribution.
For the tensor models in \cref{sec_model1,sec_model2} below, the performance of methods for solving TAPs is evaluated as follows:
\begin{enumerate}
 \item Randomly sample an order-$d$ rank-$r$ decomposition $\tuple{p} \in \Var{S}^{\times r}$ from one of the models, and let $\tensor{A} := \Phi(\tuple{p})$;
 \item create a perturbed tensor $\tensor{B} := \frac{\tensor{A}}{\|\tensor{A}\|} + 10^{-e} \frac{\tensor{E}}{\|\tensor{E}\|}$, where $\tensor{E} \sim \Var{N}$ and $e >0$;
 \item randomly sample factor matrices $\tuple{q} = (A_1,\ldots,A_d) \in \Var{E} := \R^{n_1 \times r} \times \cdots \times \R^{n_d \times r}$ with $A_k \sim \Var{N}$; and
 \item solve the TAP for $\tensor{B}$ from this random starting point $\tuple{q}$ with each method.
\end{enumerate}
Our main performance criterion is the \emph{expected time to success} (ETS), which measures how much time must be spent on average to find a solution of the TAP with method M by trying multiple random starting points if necessary. Let $t_{\mathrm{success}}$ denote the average time (in seconds) a successful attempt takes to compute a solution of the TAP, and let $p_{\mathrm{success}}$ denote the fraction of successful attempts. Similarly, write $t_{\text{fail}}$ for the average time (in seconds) needed for an unsuccessful attempt. Then,
\[
 \text{ETS}_\text{M} := {\sum_{k=0}^\infty p_{\mathrm{success}} \, p_{\text{fail}}^k (t_{\mathrm{success}} + k \cdot t_{\text{fail}})  = \frac{p_{\text{fail}} \, t_{\text{fail}} + p_{\mathrm{success}} \, t_{\mathrm{success}} }{p_{\mathrm{success}}}.}
\]
We call $\tuple{p}' \in \Var{S}^{\times r}$ a \emph{(local) solution} of the TAP if its residual is within $10\%$ of the error level, i.e., if $\|\tensor{A} - \Phi(\tuple{p}')\| \le 1.1 \cdot 10^{-e}$, and its condition number is within a factor $50$ of the condition number of the CPD $\tuple{p}$, i.e., $\kappa(\tuple{p}') \le 50 \kappa(\tuple{p})$.
In our comparisons, we present the \emph{speedup} ${{\text{ETS}}_{\text{M}}}/{{\text{ETS}}_{\text{RGN-HR}}}$ of RGN-HR relative to method M.

The ETS was {empirically} estimated for each of the methods by estimating $p_{\mathrm{success}}$, {$t_{\text{fail}},$ and $t_{\mathrm{success}}$} via their corresponding sample statistics; we sample only one tensor from the model (in step 1), one perturbation (in step 2), and $k \in \mathbb{N}$ random starting points (in step 3). For the models in \cref{sec_model1,sec_model2} we took $k = 25$ and $k = 50$, respectively.

\subsection{Parameter choices}
Both in RGN-HR and Tensorlab some parameters related to the stopping criterion should be selected. The \verb|nls_gndl| method has the same $4$ standard parameters as RGN-HR, described in \refsec{sec_stopping_criterion}, namely $\tau_f$, $\tau_{\Delta f}$, $\tau_{\Delta x}$, and $k_{\max}$. In addition, for RGN-HR we have a maximum number of restarts $r_{\max}$, and for the \verb|nls_gndl| with \verb|LargeScale = true| we have a relative tolerance on the improvement of the residual in the LSQR method $\tau_{\text{CG}}$, and a maximum number of iterations of the LSQR method $k_{\text{CG}}$.
In every experiment, $\tau_f = 0$, $\tau_{\text{CG}} = 10^{-6}$, $k_{\text{CG}} = 75$, and $r_{\max} = 500$ was selected. Furthermore, we selected $\tau_{\Delta f} = 10^{-2e}$ for RGN-HR, and $\tau_{\Delta f} = 10^{-2(e+1)}$ for the Tensorlab methods. We suggest this last distinction, as we generally observed that Tensorlab's ETS would be worse if we took $10^{-2e}$, as it negatively affected its probability of success.

\subsection{Model 1}\label{sec_model1}
The first model is a family {$\Var{F}_{r}(c,s)\subset \R^{15\times r} \times \R^{15\times r}\times \R^{15\times r}$} with $2$ parameters in which the factor matrices have correlated columns, controlled by $0 \le c < 1$, and the norms of the rank-$1$ terms approximately satisfy an exponential increase from about $1$ to about $10^s$, where $s \in \mathbb{N}$. Let $R_c$ be the upper triangular factor in the Cholesky decomposition $R_c^T R_c := c \vect{1} \vect{1}^T + (1-c) I$, where $\vect{1} \in \R^r$ is the vector of ones. Then,
\begin{multline*} 
 \Var{F}_{r}(c,s) := \bigl\{ (A_1, A_2, A_3) \in \Var{E} \;|\;
 A_k = N_k \, R_c \, \operatorname{diag}( 10^{\frac{s}{3r }}, 10^{\frac{2 s}{3r }}, \ldots, 10^{\frac{rs}{3r }} ), \\ \text{ where } N_k \sim \Var{N},\; k = 1, 2, 3 \bigr\}.
\end{multline*}
We observed empirically that for fixed $r,s$, increasing $c$ increases the condition number of the CPD. For $(r,c,s,e) \in \{ 15, 20, 25, 30 \} \times \{0, \tfrac{1}{4}, \tfrac{1}{2}, \tfrac{3}{4}, 0.95\} \times \{1, 2, 3, 4\} \times \{3,5,7\}$, we sampled a random CPD from $\Var{F}_{r}(c,s)$, applied one random perturbation (as in step 2 of the experimental setup), generated $k = 25$ random starting points from whence each of the optimization methods starts, and record the time and whether the CPD was a solution. The ETS is estimated from these data. The tested methods were RGN-HR, a variant of RGN-HR called RGN-Reg {described below}, GNDL, and GNDL-PCG.\footnote{In fact, we also tested Tensorlab's \texttt{nls\_lm} with both direct and iterative solves, i.e., with the \texttt{LargScale} option set alternatively to \texttt{false} and \texttt{true}. It was not competitive with the methods in \cref{fig_model1,fig_model3}, successfully solving only $31\%$, respectively $58\%$, of the combinations.} We took $\tau_{\Delta x} = 10^{-12}$ and $k_{\max} = 1500$. The speedups of RGN-HR with respect to the competing methods are shown in \reffig{fig_model1}. A value $\alpha$ means that the corresponding method has an ETS that is $\alpha$ times the ETS of RGN-HR; $\alpha=\infty$ means that the competing method could not solve the TAP. 

\begin{figure}[p] \centering\small
\vspace{-1em}
\begin{subfigure}{\textwidth}
\includegraphics[width=\textwidth]{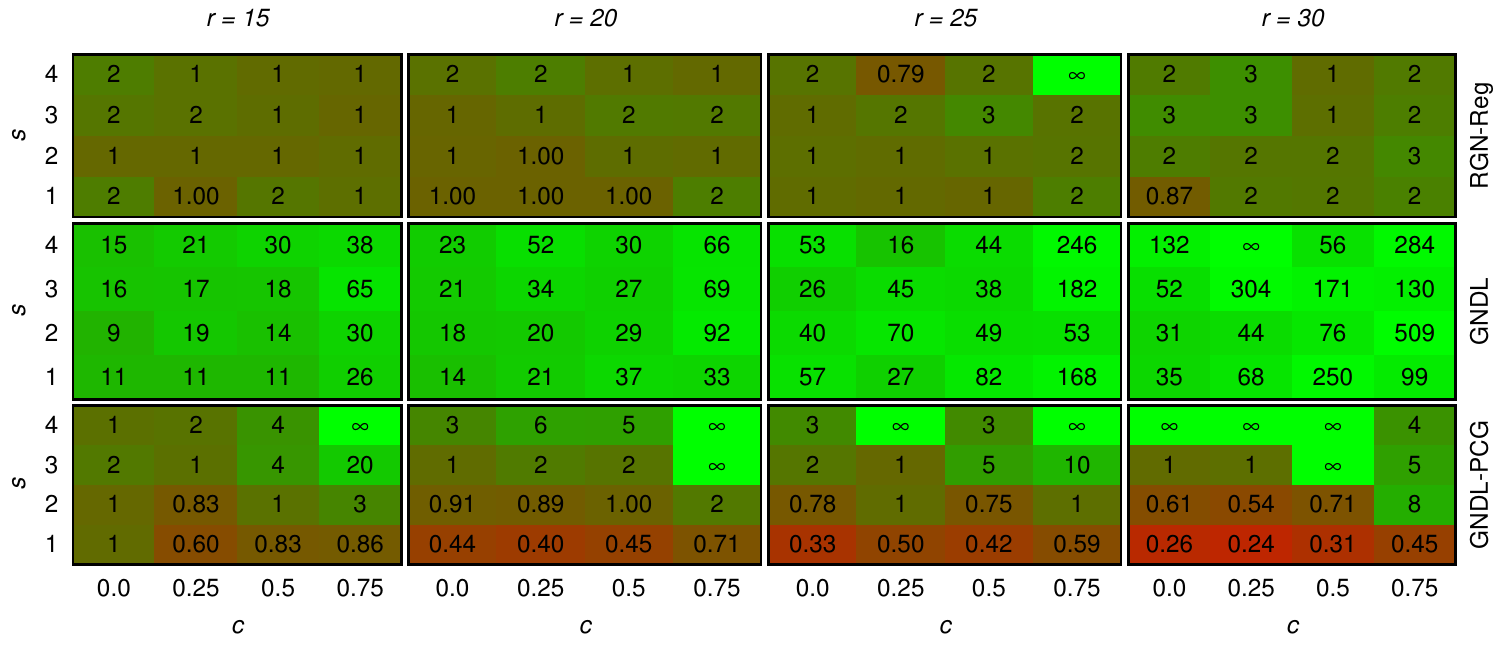}
\caption{$e = 3$}
\end{subfigure}

\begin{subfigure}{\textwidth}
\includegraphics[width=\textwidth]{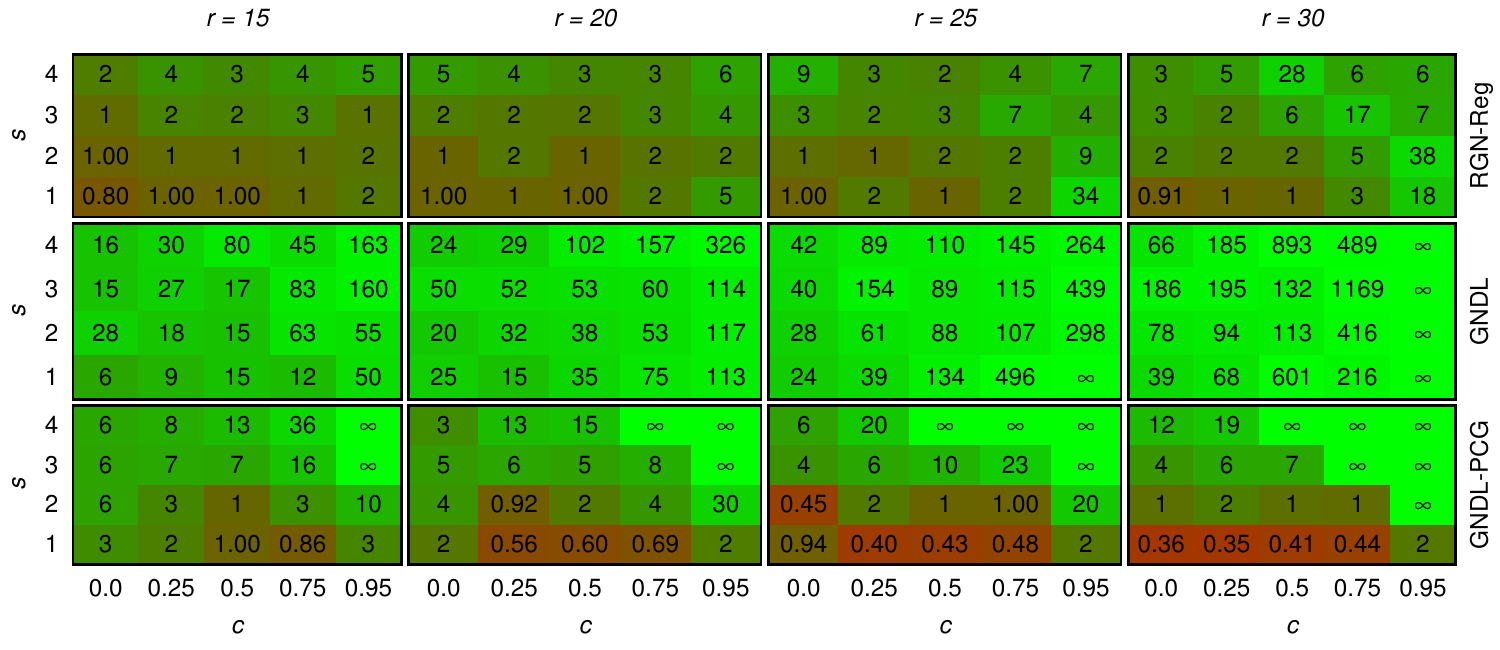}
\caption{$e = 5$}
\end{subfigure}

\begin{subfigure}{\textwidth}
\includegraphics[width=\textwidth]{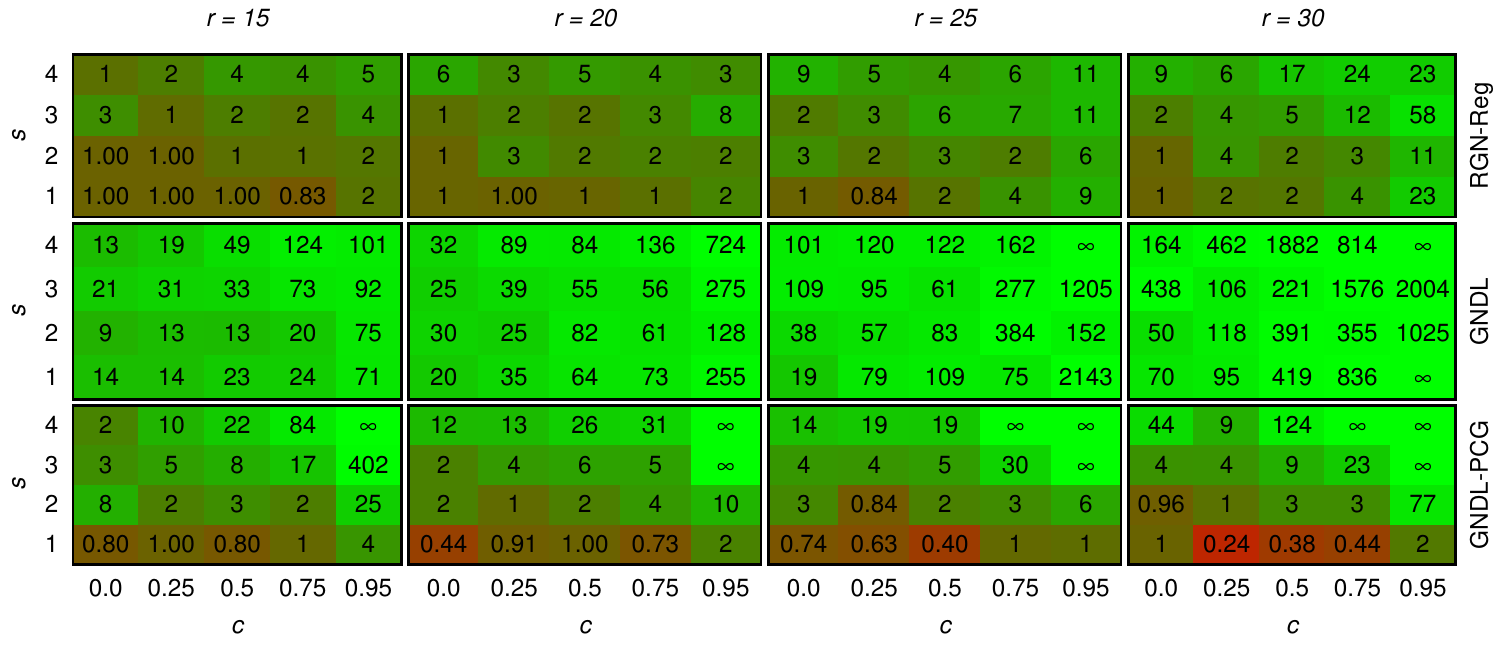}
\caption{$e = 7$}
\end{subfigure}

\caption{Speedups in terms of the ETS (based on $25$ samples) of RGN-HR with respect to RGN-Reg, and Tensorlab's GNDL and GNDL-PCG on tensors sampled from model $\Var{F}_{r}(c,s)$ with perturbation factors $e = 3, 5, 7$.} \label{fig_model1}
\end{figure}

We do not include the results for the experiment's parameter choices $e = 3$ and $c = 0.95$ because the condition number ranged between about $223$ and $1186$ for these configurations. At error level $e = 3$, this entails that the rank-$1$ terms in the computed CPD can only be guaranteed to capture between $0$ and $0.65 \approx -\log_{10}(223 \cdot 10^{-3})$ significant correct digits of the true CPD. Hence, these ``results'' are meaningless.

RGN-Reg is a variation of RGN-HR without hot restarts. Ill-conditioned normal equations are resolved by applying Tikhonov regularization:
\[
 \Biggl(H_{\tuple{p}} + 10^{-10} \Bigl(\frac{\|\Phi(\tuple{p}) - \tensor{B}\|}{\|\tensor{B}\|} \Bigr)^{\frac{3}{4}} \| H_{\tuple{p}} \|_F \, I \Biggr) \vect{p}_{\text N} = - \vect{g}_{\tuple{p}},
\]
where the notation is as in \refeqn{eqn_GN_and_CP}.
The reason for including this method is to illustrate that the proposed hot restarts mechanism provides a superior way of handling ill-conditioned Hessian approximations. The top rows of the tables in \reffig{fig_model1}(a)--(c) show the speedup of RGN-HR with respect to RGN-Reg. In almost all configurations, the hot restarts of the former significantly outperformed the latter's regularization.

The experiments show that RGN-HR outperforms both GNDL and GNDL-PCG in this small-scale model. Specifically, it outperforms GNDL, by at least $1$ and up to $3$ orders of magnitude in almost all configurations.
This difference is partly explained by the way in which GNDL solves $J \vect{p}_{\text N} = -\vect{r}$, namely by explicitly applying the pseudoinverse of $J$. This is more expensive than solving a linear system via a Cholesky factorization.\footnote{For square matrices computing an SVD is $20$ times as expensive; see \cite[p. 78 and p. 239--240]{ANLA}.} The other gains are attributed to the reduction of the number of iterations, which can be ascribed to only two key differences between RGN-HR and GNDL, namely hot restarts for avoiding ill-conditioned CPDs altogether; and the Riemannian formulation.
Indeed, both methods use a trust region, employ dogleg steps, and feature similar stopping criteria and parameter selections.
It is particularly noteworthy that RGN-HR's performance greatly improves relative to GNDL-PCG (and GNDL) when $s$ increases. The reason is that they require more iterations to converge, while RGN-HR's number of iterations remains relatively stable. \emph{The analysis in \refsec{sec_convergence} already predicted this outcome.}%

For a very significant fraction of starting points, the state-of-the-art methods halt at extremely ill-conditioned CPDs $\tuple{p}'$. The table below shows the fraction of cases where $\kappa(\tuple{p}) > 10^e$ among those CPDs $\tuple{p}'$ whose backward error is very small, namely $\|\Phi(\tuple{p}') - \tensor{B}\|_F \le 1.1 \cdot 10^{-e}$. Since the forward error is asymptotically bounded by $\kappa(\tuple{p}) \|\Phi(\tuple{p}') - \tensor{B}\|_F \approx 1.1$, this means that such CPDs are \emph{completely uninterpretable}: no correct significant digits are present in the individual rank-$1$ terms, unless their norms would be orders of magnitude larger than $\|\Phi(\tuple{p}')\|_F$. Nowadays the backward error is still the dominant criterion to ``determine'' if a computed CPD is ``a solution'' of the TAP. The table below highlights just how dangerous this practice is:
\[\footnotesize
 \begin{array}{lrrrr}
  \toprule
  e & \text{RGN-HR} & \text{RGN-Reg} & \text{GNDL} & \text{GNDL-PCG} \\
  \midrule
  3 & 23.6\%   &	42.3\% &	51.0\% &	50.8\% \\
  5 & 0.0\%    & 	10.9\% &	9.1\%  &	9.3\% \\
  7 & 0.0\%    &	0.0\%  &	0.0\%  &	0.0\% \\
  \bottomrule
 \end{array}
\]
Specifically for tensors corrupted by large amounts of noise, i.e., small $e$, the situation is dramatic. For $e=3$, about \emph{half} of the CPDs returned by the state-of-the-art, default solver in Tensorlab v3.0 whose backward error is as good as one can hope, i.e., about $10^{-e}$, are extremely ill-conditioned. The proposed RGN-HR method explicitly attempts to avoid this problem via hot restarts. The table shows that it is fairly successful compared to the other methods, which do not use hot restarts, halving the fraction of evidently spurious CPDs to less than $\frac{1}{4}$ for $e=3$, and eliminating the issue for $e = 5, 7$. As in \cite{V2017,BV2017} we caution that a CPD can only be interpreted meaningfully if both the backward error and the condition number are small.

\subsection{Model 2}\label{sec_model2}
Another challenging family is $\Var{G}_{r}(s)\subset \R^{13\times r}\times \R^{11\times r} \times \R^{9\times r}$, where $s \ge 0$ is a parameter that controls how close the tensor is to the variety of tensors of multilinear rank componentwise bounded by $(r_1,r_2,r_3) = (5,5,5)$. Specifically,
\begin{multline*} 
 \Var{G}_{r}(s) := \{ (A_1, A_2, A_3) \in \Var{E} \;|\; A_k = N_k (10^{\frac{2-s}{2}} I_r + X_k Y_k^T ) \diag( 5^{\frac{0}{r-1}}, \ldots, 5^{\frac{r-1}{r-1}}), \\
 \text{where } X_k, Y_k \in \R^{r \times r_k} \text{ and } N_k, X_k, Y_k \sim \Var{N},\; k=1,2,3 \}.
\end{multline*}
Empirical evidence shows that for fixed $r$, letting $s \to 0$ increases the CPD's condition number; a theoretical argument was sketched in \cite[section 6.1]{BV2017}.

For each $(r,s) \in \{ 5, 7, 9, 11, 13 \} \times \{0,1,\ldots,4\}$, we sampled one CPD from $\Var{G}_{r}(s)$, applied one random perturbation with $e=5$, generated $k = 50$ random starting points from whence each of the methods starts, and recorded the time and whether the CPD was a solution. From these data, the ETS is estimated. We tested RGN-HR, GNDL, and GNDL-PCG with $\tau_{\Delta x} = 10^{-15}$ and $k_{\max} = 7500$. The speedups are shown in \reffig{fig_model3}.

\begin{figure}[tb]\centering\small
\includegraphics[width=.7\textwidth]{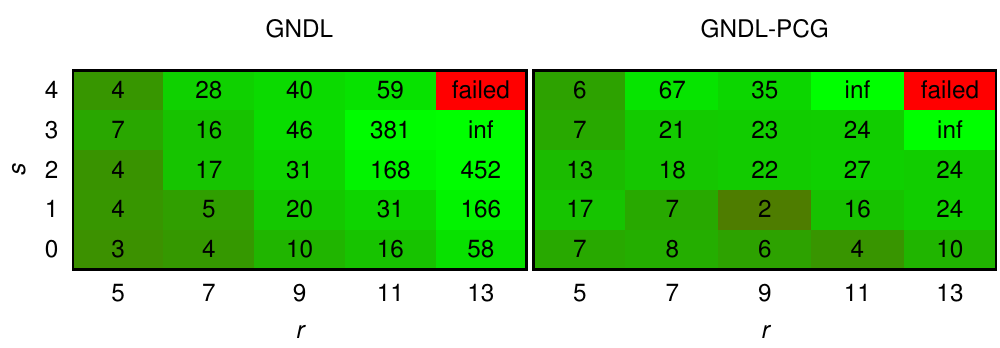}
\caption{Speedups in terms of the ETS (based on $50$ samples) of RGN-HR with respect to RGN-Reg, and Tensorlab's GNDL and GNDL-PCG on tensors sampled from model $\Var{G}_{r}(s)$ with perturbation factors $e = 5$. The label ``failed'' indicates that none of the methods could solve the problem in $50$ attempts.}
\label{fig_model3}
\end{figure}

Broadly the same observations hold as with the previous model: RGN-HR is up to $2$ orders of magnitude faster than GNDL, and up to $1$ order of magnitude faster than GNDL-PCG. In all configurations, RGN-HR was the fastest method.

\begin{figure}\centering \small
\includegraphics[width=.49\textwidth]{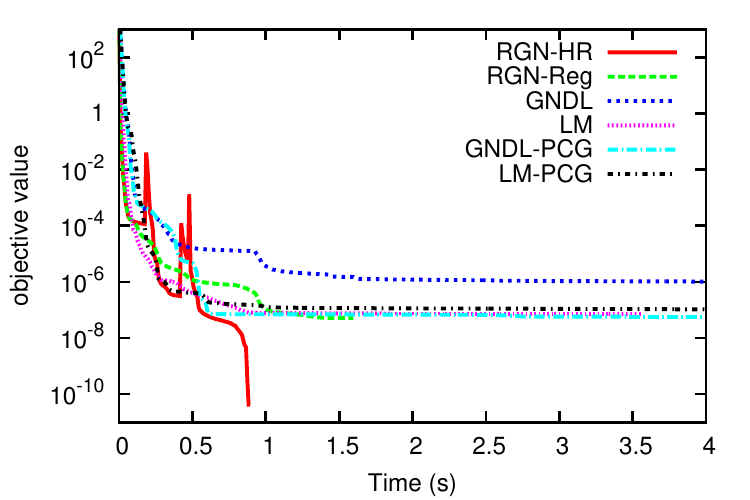}
\includegraphics[width=.49\textwidth]{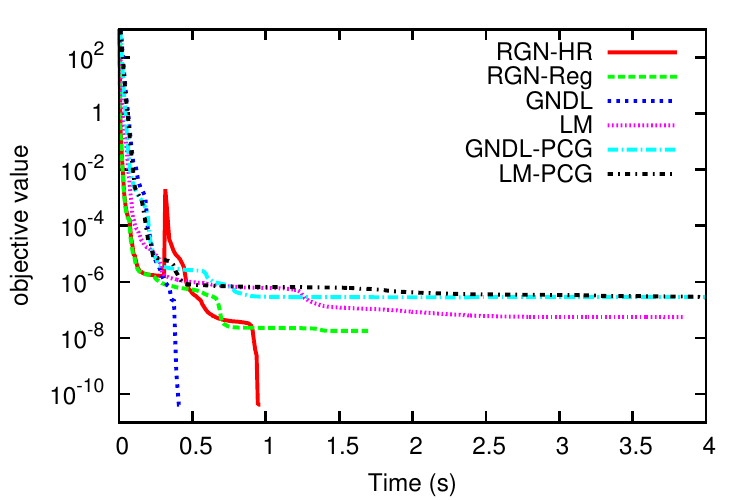} 
\includegraphics[width=.49\textwidth]{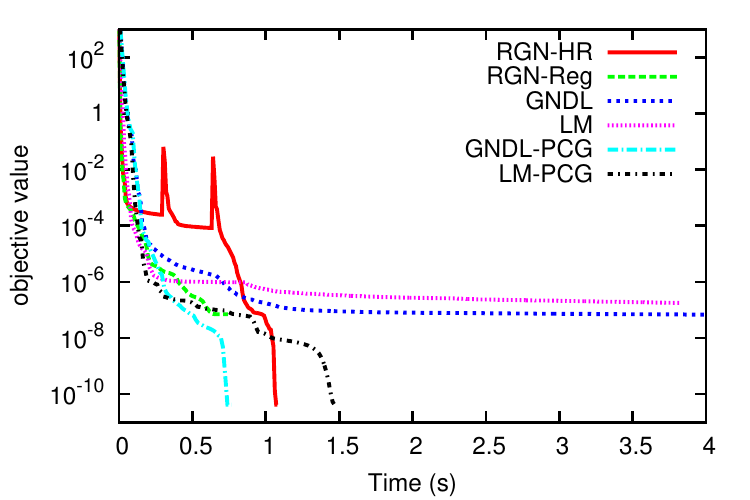}
\includegraphics[width=.49\textwidth]{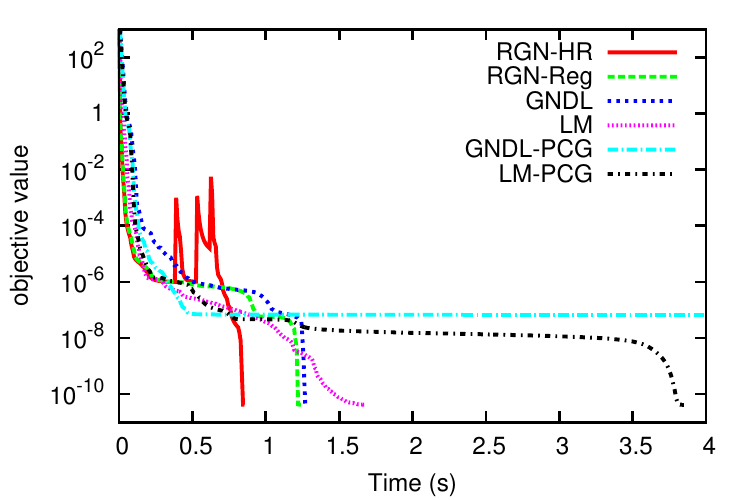}
\caption{Randomly selected convergence graphs for the model $\Var{G}_{7}(2)$ with error level $e=5$. The objective value $f(\tuple{p}) = \frac{1}{2}\|\Phi(\tuple{p}) - \tensor{B}\|_F^2$ is plotted versus the time spent for the proposed RGN-HR method, the variant RGN-Reg, and Tensorlab's GNDL, GNDL-PCG, LM and LM-PCG methods.}
\label{fig_model3_convergence}
\end{figure}

Finally, in \reffig{fig_model3_convergence} we show some convergence plots for this model with rank $r=7$ and $s=2$. It includes Tensorlab's \texttt{nls\_lm} method, a Levenberg--Marquardt method, both with direct solves and iterative solves of the least-squares problem $J \vect{p}_{\text{N}} = -\vect{r}$; we refer to them as respectively LM and LM-PCG.

The effect of the hot restarts in RGN-HR can clearly be seen in \reffig{fig_model3_convergence}. Observe in particular that a hot restart is often triggered after a brief period of stagnation of the convergence. {Contrast this with the prolonged periods of stagnation in all of the other methods, including RGN-Reg.} We surmise that stagnation is often \emph{caused} by ill-conditioned Hessian approximations. In the plots on the right we see that the progress of RGN-HR and RGN-Reg is almost identical up to $0.4$ seconds. At that point, the former detects that the condition number is too high and applies a hot restart, while the latter steadily reduces the objective value. Nevertheless, in the end, RGN-HR converges faster.

\section{An application} \label{sec_application}
The proposed algorithm can also be applied efficiently to larger data sets appearing in real applications. Fluorescence spectroscopy is an imaging technique that can be employed to detect the concentration of certain chemical compounds, called \emph{fluorophores}, in a diluted mixture. This inexpensive analysis technique is widely employed in the life sciences; for example, it is employed for the identification of dissolved organic material in natural and waste water \cite{HBR2007} and in food chemistry \cite{SBG2004}, among others. The theoretical model underlying the \textit{emission--excitation matrices} that are obtained from a fluorescence spectroscopy analysis of a diluted, inert mixture of several fluorophores is the tensor rank decomposition \cite{AD1981}.

The tensor we consider was obtained from fluorescence spectroscopy measurements of five mixtures of three amino acids; in ideal circumstances the rank of the tensor would thus be $3$.\footnote{The data can be obtained at \url{http://www.models.life.ku.dk/Amino_Acid_fluo}.} Detailed information about the acquisition of the data can be found in \cite{Bro1998,Kiers1998}. The size of the tensor is $5 \times 201 \times 61$; the first factor corresponds to the mixtures, the second to the emission wavelengths ($250$--$450$ nm in steps of $1$ nm), and the third to the excitation wavelengths ($250$--$310$ nm in steps of $1$ nm).

\begin{figure}\centering\small
 \includegraphics[width=.49\textwidth]{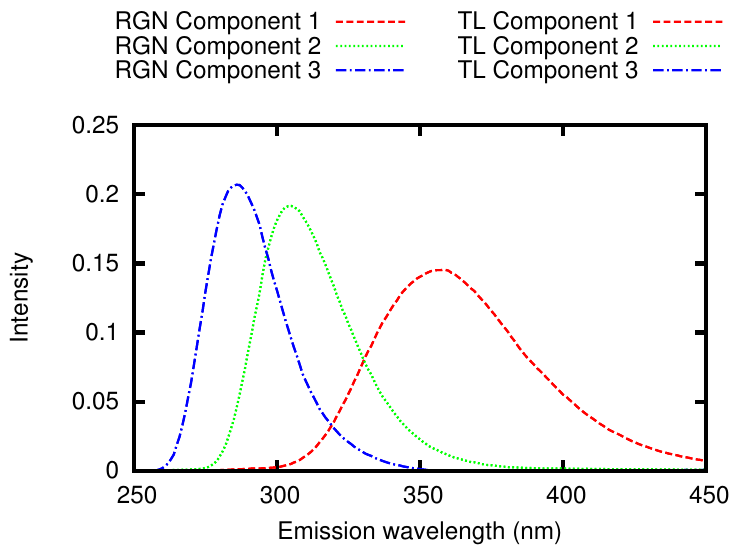}
 \includegraphics[width=.49\textwidth]{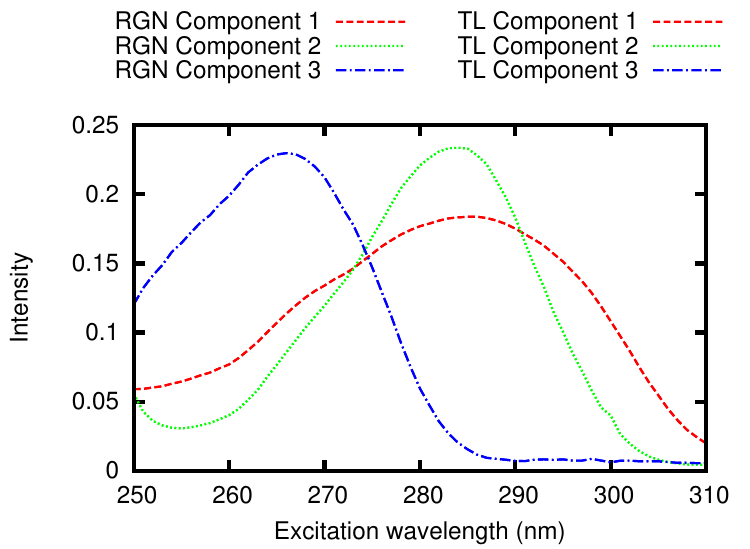}
\caption{Emission and excitation spectra of the three identified components in the mixtures of the amino acids using both the proposed algorithm (RGN) and Tensorlab's \texttt{cpd\_nls} algorithm (TL). We stress that all 12 spectra are plotted; the spectra recovered by the two methods are visually indistinguishable.}
\label{fig_real_data}
\end{figure}

As explained in \refrem{rem_tucker_compression}, it is common for large but low-rank tensors to start with a Tucker compression. We approximated $\tensor{B} \in \R^{5 \times 201 \times 61}$ by a rank-$(5,6,6)$ ST-HOSVD \cite{VVM2012} $(Q_1, Q_2, Q_3) \cdot \tensor{S} \approx \tensor{B}$ using Tensorlab's \texttt{mlsvd} function. This took about $0.01$ seconds, and results in a relative error of $1.236 \cdot 10^{-2}$. Both RGN-HR and GNDL-PCG were then applied to $\tensor{S}$ to compute a rank-$3$ approximation, taking about $0.08$ and $0.06$ seconds respectively. Let $(M_1,M_2,M_3)$ and $(M_1', M_2', M_3')$ denote their respective factor matrices. Then, $(Q_1 M_1, Q_2 M_2, Q_3 M_3)$ and $(Q_1 M_1', Q_2 M_2', Q_3 M_3')$ are the factor matrices of the corresponding rank-$3$ tensors $\tensor{B}_3$ and $\tensor{B}_3'$. The relative error between $\tensor{B}_3$ and $\tensor{B}$ was approximately equal to $2.50493697 \cdot 10^{-2}$, and likewise for $\tensor{B}_3'$. Computing the best-$3$ approximation directly from $\tensor{B}$, without Tucker compression, yields an approximation error of about $2.505 \cdot 10^{-2}$ in both cases.

The columns of the second factor matrices of $\tensor{B}_3$ and $\tensor{B}_3'$ represent the emission spectra of the $3$ identified components; they are plotted in the left graph in \reffig{fig_real_data}. The right graph in that figure visualizes the columns of the third factor matrices of $\tensor{B}_3$ and $\tensor{B}_3'$, which represent the excitation spectra. Note that the spectra recovered by RGN-HR are visually indistinguishable from those recovered by GNDL-PCG. For brevity, we do not plot the concentration profiles represented by the first factor matrices; their relative difference in Frobenius norm is approximately $2 \cdot 10^{-6}$.

The true concentrations are also provided in the data set. For both RGN-HR and GNDL-PCG, the columns of the first factor matrix correlate very well with the true data: correlation coefficients greater than $99.8\%$ were obtained in all cases.

\section{Conclusions} \label{sec_conclusions}
We proposed the first Riemannian optimization method for approximating a given tensor by one of low canonical rank. The theoretical analysis in \refsec{sec_trd} motivated why the proposed Riemannian formulation of the TAP should be preferred over the naive, overparameterized formulation involving factor matrices if a Riemannian Gauss--Newton method is employed. Specifically, our analysis predicts a great advantage of the proposed RGN method over state-of-the-art non-Riemannian GN methods when the CPD contains rank-$1$ terms whose norms are of different magnitudes. The numerical experiments in \refsec{sec_numerical_experiments} confirmed this theory.

The second main innovation explored in this paper was to exploit information about the condition number in the optimization method. It was argued in \refsec{sec_ill_H_p} that certain ill-conditioned CPDs are hard to escape with the RGN process. We proposed hot restarts for escaping such regions of ill-conditioned CPDs. The numerical experiments revealed significant speedups attributed to these hot restarts.

Based on the numerical experiments, we believe that RGN-HR can be a suitable alternative for classic GN methods with direct solves of the TRS. Speedups between~$3$ and $2000$ with respect to the state-of-the-art GN method \texttt{nls\_gndl} in Tensorlab~v3.0 \cite{Tensorlab} were observed. Our experiments suggest that RGN-HR can be competitively applied up to $r(\Sigma+d) \approx 1000$, which is one order of magnitude larger than Tensorlab's default choice for \texttt{nls\_gndl}.

An obstacle for extending the RGN-HR method to large-scale TAPs in which the least-squares problem in the TRS is only approximately solved, for example via LSQR, is how a good estimate of the condition number \refeqn{kappa} should be computed inexpensively. As RGN-HR was very competitive for small-scale problems, we believe that this is a promising direction for future work. Another interesting problem is to prove the observations in the informal analysis of \refsec{sec_ill_H_p} rigorously.

\section*{Acknowledgments}
We thank B.~Vandereycken for helpful feedback he shared with us at the $7$th Workshop on Matrix Equations and Tensor Techniques in Bologna, Italy held in February 2017. We also thank two anonymous referees for their valuable feedback that improved this article.

\bibliographystyle{siamplain}

\appendix
\addtolength{\abovedisplayskip}{-1.75pt}
\addtolength{\belowdisplayskip}{-1.75pt}

\section{Implementation details}\label{app_implementation}
This appendix describes efficient formulations of the critical steps in \refalg{alg_rlm}. It is shown below that the following number of operations per iteration (lines 6--11) are obtained:
\begin{equation*}
\begin{array}{ll}
\toprule
\text{line 6:} & \mathcal{O}(dr\Pi + r\Sigma + d^2 r^2 \Sigma^2) \\
\text{line 7:} & \mathcal{O}( r^3 (\Sigma+1)^3 + (r+1)\Pi ) \\
\text{line 8:} & \mathcal{O}(4r\Sigma + r 2^{d+1}) \\
\text{lines 9--11:} & \mathcal{O}(\Pi + r\Sigma) \\
\bottomrule
\end{array}
\end{equation*}
\vspace{.1em}

\subsection{Retraction} \label{sec_fast_retraction}
As observed in \cite[section 3.3]{KSV2014}, the T-HOSVD retraction is computed efficiently by exploiting the structure of the tangent vectors. Essentially the same observations are valid for the ST-HOSVD retraction with foot at the point $p_i = \alpha_i \sten{a}{i}{1} \otimes \cdots \otimes \sten{a}{i}{d}$, where $\sten{a}{i}{k} \in \mathbb{S}^{n_k-1}$. For completeness, we explain this below.

The TRS \refeqn{eqn_tr_model} is solved in local coordinates with respect to the basis $B$ in \refeqn{eqn_orth_basis}, yielding $\widehat{\vect{p}}$ as solution. Write
\(
\widehat{\vect{p}}^T =
\begin{bmatrix}
\sten{x}{1}{T} &
\cdots &
\sten{x}{r}{T}
\end{bmatrix}
\)
with $\vect{x}_i \in \R^{\Sigma+1}$, and then the components of the tangent vector $\vect{t} = (\vect{t}_1, \ldots, \vect{t}_r) \in \Tang{\tuple{p}}{\Var{S}^{\times r}}$ are
\begin{align*}
\vect{t}_i
= T_{p_i} \sten{x}{i}{}
= \sum_{k=1}^d T_{i,k} \, \sten{x}{i}{k} \quad\in\Tang{p_i}{\Var{S}}.
\end{align*}
Since $U_{i,1} = I$, we find by the multilinearity of the tensor product that
\begin{multline*}
p_i + \vect{t}_i = (\sten{x}{i}{1} +  \alpha_i \sten{a}{i}{1}) \otimes \sten{a}{i}{2}\otimes\cdots\otimes\sten{a}{i}{d} \\
+\sten{a}{i}{1}\otimes (U_{i,2} \sten{x}{i}{2})\otimes \sten{a}{i}{3} \otimes \cdots\otimes \sten{a}{i}{d} +
 \cdots + \sten{a}{i}{1}\otimes\cdots\otimes\sten{a}{i}{d-1}\otimes (U_{i,d}\sten{x}{i}{d}).
\end{multline*}
Let $\beta_{i,k} \sten{q}{i}{k} = U_{i,k}\sten{x}{i}{k}$ with $\|\sten{q}{i}{k}\|=1$ for $k=2,\ldots,d$. Also define $\beta_{i,1} \sten{q}{i}{1} = \sten{x}{i}{1} + \alpha_i \sten{a}{i}{1}$ with $\|\sten{q}{i}{1}\|=1$.
Then, we may write $p_i + \vect{t}_i$ as {the orthogonal Tucker decomposition}
\begin{align}\label{eqn_rank_one_plus_tg}
 p_i + \vect{t}_i = ( \sten{q}{i}{1}, Q_2, \ldots, Q_d ) \cdot \tensor{S}, \text{ where } Q_k = \begin{bmatrix} \sten{a}{i}{k} & \sten{q}{i}{k} \end{bmatrix} \text{ for }  k = 2,\ldots,d,
\end{align}
and where the order-$d$ tensor $\tensor{S} \in \R^{1 \times 2 \times \cdots \times 2}$ is given by
\[
 s_{1,i_2,\ldots,i_d} =
 \begin{cases}
 \beta_{i,1} & \text{if } i_2 = \cdots = i_d = 1, \\
 \beta_{i,k} & \text{if } i_k=2 \text{ and } i_\ell = 1 \text{ for } 2 \le \ell \ne k \le d,\\
  0 & \text{otherwise}.
 \end{cases}
\]
Note that $Q_k^T Q_k = I$, by definition of $U_{i,k}$ in \refeqn{eqn_def_tangent_segre_repr}. To complete the derivation, we need the following result about ST-HOSVD's, which is proved in \refapp{app_proof_sthosvd_retraction}.

\begin{lemma}\label{lem_sthosvd_compression}
Let $\tensor{A} \in \R^{n_1 \times \cdots \times n_d}$ admit an orthogonal Tucker decomposition $\tensor{A} = (Q_1, \ldots, Q_d) \cdot \tensor{B}$, where $Q_k \in \R^{n_k \times m_k}$ has orthonormal columns and $\tensor{B} \in \R^{m_1 \times \cdots \times m_d}$.
Let $(Z_1, \ldots, Z_d) \cdot \tensor{C}$ be a rank-$(r_1,\ldots,r_d)$ ST-HOSVD approximation of $\tensor{B}$ corresponding to the processing order $\pi$. Then, $(Q_1 Z_1, \ldots, Q_d Z_d) \cdot \tensor{C}$ is a rank-$(r_1,\ldots,r_d)$ ST-HOSVD approximation of $\tensor{A}$ corresponding to $\pi$.
\end{lemma}

The rank-$(1,\ldots,1)$ ST-HOSVD approximation of $p_i + \vect{t}_i = (\sten{q}{i}{1}, Q_2,\ldots,Q_d) \cdot \tensor{S}$ in \refeqn{eqn_rank_one_plus_tg} can thus be computed efficiently from the ST-HOSVD approximation $(1,\sten{r}{i}{2},\ldots,\sten{r}{i}{d}) \cdot \lambda$ of $\tensor{S}$; it is namely $(\sten{q}{i}{1}, Q_2 \sten{r}{i}{2}, \ldots, Q_d \sten{r}{i}{d}) \cdot \lambda$.

Computing the retraction as above requires the following number of operations per rank-$1$ tensor $p_i \in \Var{S}$:
$\mathcal{O}(2\Sigma)$ operations for constructing $\sten{q}{i}{1}$ and $Q_2$, \ldots, $Q_d$; $\mathcal{O}(2^{d+1})$ operations for computing the ST-HOSVD approximation of $\tensor{S}$; and $\mathcal{O}(2 \Sigma)$ operations for recovering the ST-HOSVD approximation of $p_i + \vect{t}_i$ from the approximation of $\tensor{S}$. This results in
\(
 \mathcal{O}( r(4 \Sigma + 2^{d+1}) )
\)
operations for one product ST-HOSVD retraction.

\subsection{Gradient and Hessian approximation} \label{sec_fast_hessian} \label{sec_fast_gradient}
Recall that the set of tensors of rank bounded by $r$ is the image of the map ${\text\textlbrackdbl} \cdot {\text\textrbrackdbl}$ in \refeqn{factor_matrices_def}. Its Jacobian is $J$ as in~\refeqn{eqn_overparam_jacobian}. Let $\tuple{p} = (\sten{a}{i}{1}\otimes\cdots\otimes\sten{a}{i}{d})_{i=1}^r$ and write $D:=\operatorname{diag}(U_{1,1},\ldots,U_{r,1},\ldots,U_{1,d},\ldots,U_{r,d})$, where the $U_{i,j}$ are defined as in \refsec{sec_parametrization}.
Comparing with \refeqn{eqn_jacobian}, we note \(
 T_{\tuple{p}} = J D
\).

For $\vect{x} \in \R^\Pi$, the operation $J^T \vect{x}$ is an important computational kernel called the \emph{CP gradient} \cite{Phan2013}. Computationally efficient implementations were proposed in \cite{Phan2013,VMV2015}. Since $T_{\tuple{p}}^T \vect{g}_{\tuple{p}}$ is equivalent to $D^T (J^T \vect{g}_{\tuple{p}})$, one can employ any of the efficient methods for computing CP gradients; afterwards multiply the result by the block diagonal matrix $D$. In our Matlab implementation, the CP gradient is computed via left-to-right and right-to-left (RTL) contractions \cite{Phan2013}, \cite[section 2.3]{VMV2015}. This scheme requires $\mathcal{O}(d r\Pi + r\Sigma)$ operations.

Efficient algorithms for computing $J^T J$ were investigated in the literature; e.g., \cite{Sorber2013a,TDtB2011,Paatero1997,TomasiPhD,Phan2013a}. One can then efficiently compute $H_{\tuple{p}}$ as
\(
H_{\tuple{p}} = T_{\tuple{p}}^T T_{\tuple{p}} = D^T (J^T J) D.
\)
In our implementation we chose the method from \cite{Sorber2013a} for constructing $J^T J$ efficiently. The computational complexity of this algorithm is approximately $\mathcal{O}\bigl( (d r \Sigma)^2 \bigr)$.

\subsection{Optimal coefficients} \label{sec_optimal_coefficients}
The solution $\vect{x}^*$ of least-squares problem \refeqn{eqn_optimal_coefficients} is obtained by observing that it is equivalent to
\[
\vect{x}^* = \begin{bmatrix} \sten{a}{1}{1}\otimes\cdots\otimes\sten{a}{1}{d} & \cdots & \sten{a}{r}{1}\otimes\cdots\otimes\sten{a}{r}{d} \end{bmatrix}^\dagger \operatorname{vec}(\tensor{B}) =: (A_1 \odot \cdots \odot A_d)^\dagger \vect{b},
\]
where $\vect{b} = \operatorname{vec}(\tensor{B}) \in \R^\Pi$ is the vectorization of $\tensor{B}$, $A_k = \bigl[ \sten{a}{i}{k} \bigr]_{i=1}^r$ are the factor matrices and $\odot$ is the columnwise Khatri--Rao product.
Recall that the optimal coefficients are computed either in step {2} of \refalg{alg_rlm} after randomly initializing the rank-$1$ tensors, or in step 11 of \refalg{alg_hot_restarts} after randomly perturbing them.
As a result, in both cases $A_1 \odot \cdots \odot A_d$ has linearly independent columns with probability~$1$.\footnote{One of our assumptions is that $r < \frac{\Pi}{\Sigma+1}$ is strictly subgeneric, so that in particular $r < \Pi$.}
Then, it is well-known \cite{KB2009} that
\[
 \vect{x}^* = (A_1 \odot \cdots \odot A_d)^\dagger \vect{b} = \bigl( (A_1^T A_1) \circledast \cdots \circledast (A_d^T A_d) \bigr)^{-1} \bigl( (A_1 \odot \cdots \odot A_d)^T \vect{b} \bigr),
\]
where $\circledast$ is the element-wise or Hadamard product.
The rightmost matrix-vector product can be interpreted as $r$ simultaneous tensor-to-vector contractions, which we compute with RTL contractions. Thereafter, the linear system is constructed as the formula suggests and solved via a Cholesky factorization. The optimal coefficients can thus be computed in $\mathcal{O}(r\Pi + r^2 \Sigma + r^3)$ operations.

\section{Proofs of the lemma's}\label{app_proof_sthosvd_retraction}

\begin{proof}[Proof of \reflem{norm_of_phi}]
We have $\deriv{\tuple{p}}{\Phi}(U_1{\vect{x}}_1, \ldots, U_r {\vect{x}}_r) = U_1 {\vect{x}}_1 + \ldots + U_r {\vect{x}}_r$ from \cite{BV2017}, where ${\vect{x}}_i \in \R^{\Sigma+1}$, and $U_i \in \R^{\Pi \times (\Sigma+1)}$ contains an orthonormal basis of $\Tang{p_i}{\Var{S}}$. Then, using the triangle inequality we obtain
\begin{align*}
 \|\deriv{\tuple{p}}{\Phi}\|_2
 = \max_{\|U_1 {\vect{x}}_1\|^2 + \cdots + \|U_r {\vect{x}}_r\|^2 = 1} \|U_1 {\vect{x}}_1 + \cdots + U_r {\vect{x}}_r\|
 \le \max_{c_1^2 + \cdots + c_r^2 = 1} \bigl( c_1 + \cdots + c_r \bigr),
\end{align*}
where we set $c_i := \|U_k \vect{x}_i \| = \|{\vect{x}}_i\| \in \R^r$ in the last step. Since we have the inequality $\max_{\|\vect{c}\|=1} \|\vect{c}\|_1 \le \sqrt{r} \|\vect{c}\| = \sqrt{r}.$ This proves the upper bound.

To prove the lower bound, take any $\vect{x}_1$ of unit norm. Then,
$$\Vert\deriv{\tuple{p}}{\Phi}(U_1{\vect{x}}_1, U_20, \ldots, U_r 0)\Vert = \Vert U_1\vect{x}_1\Vert = 1,$$
which implies that $\|\deriv{\tuple{p}}{\Phi}\|_2 \geq 1$. The proof is concluded.
\end{proof}

\begin{proof}[Proof of \reflem{prop_I_r}]
Let $\Var{S}_{\C}$ denote the complexification of $\Var{S}$, and let $\sigma_r(\Var{S}_{\C})$ denote the $r$-secant variety of $\Var{S}_{\C}$ \cite{Harris1992}.
Recall that the real points of $\mathcal{S}_\C$ are dense in the Zariski topology \cite[section 5]{QCL2016}, that $\sigma_r(\mathcal{S})$ (and so $\sigma_r(\mathcal{S}_\C)$ as well) is non-defective by assumption, and, hence, that $\operatorname{rank} T_{\tuple{p}} = \dim \mathrm{T}_{\Phi(\tuple{p})} \sigma_r(\mathcal{S}_\C) = \dim (\Var{S}_\C)^{\times r}$ for all $\tuple{p}\in\Var{S}^{\times r}$ in a Zariski-dense set by Terracini's lemma \cite{Landsberg2012}.
Then, applying Sard's theorem \cite{Harris1992} to $\Phi_\C : (\Var{S}_{\C})^{\times r} \to \Var{S}_{\C}, (p_1,\ldots,p_r) \mapsto p_1 + \cdots + p_r$ shows that $\mathcal{I}_r$ is contained in the complex subvariety of critical points of $\Phi_\C$.
\end{proof}

\begin{proof}[Proof of \reflem{lem_sthosvd_compression}]
Without loss of generality we assume that $\vect{p} = [1\; 2\; \cdots\; d]$.

Let $(U_1,\ldots,U_d) \cdot \tensor{S}$ be the ST-HOSVD approximation of $\tensor{A}$. We prove by induction that $U_i = Q_i Z_i$. It suffices to show that the bases can be chosen to be the same because they completely determine the core tensor by the relationship $\tensor{S} = (U_1,\ldots,U_d)^T \cdot \tensor{A}$.

Let $USV^T$ be a compact SVD of $\tensor{A}_{(1)}$ and let $\widetilde{U}\widetilde{S}\widetilde{V}^T$ be an SVD of $\tensor{B}_{(1)}$. Then,
\begin{align*}
 USV^T
 = \tensor{A}_{(1)}
 = Q_1 \tensor{B}_{(1)} (Q_2 \otimes \cdots \otimes Q_d)^T
 = (Q_1 \widetilde{U}) \widetilde{S} (\widetilde{V}^T (Q_2 \otimes \cdots \otimes Q_d)^T).
\end{align*}
Both the leftmost and rightmost expressions specify an SVD of $\tensor{A}_{(1)}$. Since it is es\-sentially unique,  {$\widetilde{U}$ can be chosen} so that $U = Q_1 \widetilde{U}$. By definition of the ST-HOSVD in \cite{VVM2012}, $U = [U_1 \;\; X]$ and $\widetilde{U} = [Z_1 \;\; Y]$, so that $U_1 = Q_1 Z_1$, proving the base case.

Assume now that $U_\ell = Q_\ell Z_\ell$ for all $\ell = 1, \ldots, k-1$, then we prove that it holds for $k$ as well. Define
\begin{align*}
 \tensor{A}_{(k)}^{(k-1)} &:= \tensor{A}_{(k)} (U_1 \otimes \cdots \otimes U_{k-1} \otimes I \otimes \cdots \otimes I) \text{ and }\\
 \tensor{B}_{(k)}^{(k-1)} &:= \tensor{B}_{(k)} (Z_1 \otimes \cdots \otimes Z_{k-1} \otimes I \otimes \cdots \otimes I).
\end{align*}
Since $\tensor{A}_{(k)} = Q_k \tensor{B}_{(k)} (Q_1 \otimes \cdots Q_{k-1} \otimes Q_{k+1} \otimes \cdots \otimes Q_d )^T$, it follows that
\begin{align*}
 \tensor{A}_{(k)}^{k-1}
 &= Q_k \tensor{B}_{(k)} (Q_1^T U_1 \otimes \cdots \otimes Q_{k-1}^T U_{k-1} \otimes Q_{k+1}^T \otimes \cdots \otimes Q_d^T) \\
 &= Q_k \tensor{B}_{(k)} (Z_1 \otimes \cdots \otimes Z_{k-1} \otimes Q_{k+1}^T \otimes \cdots \otimes Q_d^T) \\
 &= Q_k \tensor{B}_{(k)}^{(k-1)} (I\otimes \cdots \otimes I\otimes Q_{k+1} \otimes \cdots \otimes Q_d)^T.
\end{align*}
Let $USV^T$ be the compact SVD of $\tensor{A}_{(k)}^{(k-1)}$ and let $\widetilde{U}\widetilde{S}\widetilde{V}^T$ be the compact SVD of $\tensor{B}_{(k)}^{(k-1)}$. Then,
\[
 USV^T = \tensor{A}_{(k)}^{(k-1)} = (Q_k \widetilde{U}) \widetilde{S} (\widetilde{V}^T (I\otimes \cdots \otimes I\otimes Q_{k+1} \otimes \cdots \otimes Q_d)^T).
\]
The leftmost and rightmost expressions are both compact SVDs, hence $\widetilde{U}$ can be chosen so that $U = Q_k \widetilde{U}$. From the definition of the ST-HOSVD in \cite{VVM2012} it again follows that $U_k = Q_k Z_k$, which concludes the proof.
\end{proof}

\end{document}